\documentclass[twoside,a4paper,12pt,reqno]{amsart}

\usepackage{amssymb,amsmath,color,psfrag,graphicx}

\usepackage{caption}

\DeclareSymbolFont{euleroperators}{U}{eur}{m}{n}
\SetSymbolFont{euleroperators}{bold}{U}{eur}{b}{n}

\makeatletter
\renewcommand{\operator@font}{\mathgroup\symeuleroperators}
\makeatother

\definecolor{refkey}{rgb}{1,0,0}

\definecolor{labelkey}{rgb}{0,0,1}

\newtheorem{thm}[equation]{Theorem}

\newtheorem{lem}[equation]{Lemma}
\newtheorem{prp}[equation]{Proposition}

\theoremstyle{definition}
\newtheorem{dfn}[equation]{Definition}

\theoremstyle{remark}
\newtheorem{rem}[equation]{Remark}

\newcommand{\thmref}[1]{Theorem~\ref{#1}}
\newcommand{\prpref}[1]{Proposition~\ref{#1}}
\newcommand{\lemref}[1]{Lemma~\ref{#1}}

\newcommand{\dfnref}[1]{Definition~\ref{#1}}
\newcommand{\remref}[1]{Remark~\ref{#1}}

\newcommand{\figref}[1]{Figure~\ref{#1}}
\newcommand{\secref}[1]{Section~\ref{#1}}

\renewcommand\a{\alpha}

\renewcommand\b{\beta}
\newcommand\bnd{\operatorname{bnd}}

\newcommand\C{\mathcal C}

\renewcommand\d{\delta}
\newcommand\D{\mathcal D}

\newcommand\e{\varepsilon}

\newcommand\F{\Phi}
\newcommand\f{\varphi}
\newcommand\Fun{\operatorname{Fun}}
\newcommand\fun{\operatorname{fun}}

\newcommand{\G}{\Gamma}
\renewcommand{\gg}{\boldsymbol{g}}
\newcommand{\g}{\gamma}
\newcommand{\gr}{\operatorname{gr}}
\newcommand{\Gr}{\mathcal{G}}

\renewcommand{\k}{\varkappa}

\renewcommand\L{\mathcal{L}}
\newcommand{\la}{\lambda}
\newcommand{\lan}{\langle}

\newcommand{\mapstoto}{\mathop{\,\sim\joinrel\rightsquigarrow\,}}

\newcommand{\ov}{\overline}

\renewcommand\P{\mathbf P}
\newcommand{\p}{\partial}

\newcommand\Q{\mathbf Q}

\newcommand\R{\mathbb R}
\newcommand{\ran}{\rangle}

\newcommand\s{\sigma}

\newcommand{\sgr}{\operatorname{sgr}}
\newcommand{\supp}{\operatorname{supp}}

\newcommand\T{\mathcal T}
\renewcommand\t{\tau}
\renewcommand\th{\theta}
\newcommand{\toto}{\mathop{\,\longrightarrow\,}}

\newcommand\vn{\varnothing}

\newcommand\w{\omega}
\newcommand\wh{\widehat}
\newcommand\wt{\widetilde}

\newcommand\Z{\mathbb Z}
\newcommand\Zd{\Z\left[\frac12\right]}
\newcommand\ZZ{{\mathbb Z_+}}

\begin{document}

\title{Thompson's group $F$ is not Liouville}

\thanks{The author was partially supported by funding from the Canada Research Chairs
program, from NSERC (Canada) and from the European Research Council within European Union Seventh Framework Programme (FP7/2007-2013), ERC grant agreement 257110-RAWG. A part of this work was conducted during the trimester ``Random Walks and Asymptotic Geometry of Groups'' in 2014 at the Institut Henri Poincar\'e, Paris.}

\author{Vadim A. Kaimanovich}

\address{Department of Mathematics and Statistics, University of
Ottawa, 585 King Edward, Ottawa ON, K1N 6N5, Canada}

\email{vkaimano@uottawa.ca, vadim.kaimanovich@gmail.com}

\dedicatory{Dedicated to Wolfgang Woess on the occasion of his 60$^{\,th}$ birthday}

\begin{abstract}
We prove that random walks on Thompson's group~$F$ driven by strictly non-degenerate finitely supported probability measures $\mu$ have a non-trivial Poisson boundary. The proof consists in an explicit construction of two non-trivial $\mu$-boundaries. Both of them are described in terms of the ``canonical'' Schreier graph $\G$ on the dyadic-rational orbit of the canonical action of $F$ on the unit interval (in fact, we consider a natural embedding of $F$ into the group $PLF(\R)$ of piecewise linear homeomorphisms of the real line, and realize $\G$ on the dyadic-rational orbit in $\R$). However, the definitions of these $\mu$-boundaries are quite different (in perfect keeping with the ambivalence concerning amenability of the group~$F$). The first $\mu$-boundary is similar to the boundaries of the lamplighter groups: it consists of $\Z$-valued configurations on $\G$ arising from the stabilization of logarithmic increments of slopes along the sample paths of the random walk. The second $\mu$-boundary is more similar to the boundaries of the groups with hyperbolic properties as it consists of sections (``end fields'') of the end bundle of the graph $\G$, i.e., of the collections of the limit ends of the induced random walk on~$\G$ parameterized by all possible starting points. The latter construction is more general than the former one, and is actually applicable to any group which has a transient Schreier graph with a non-trivial space of ends.
\end{abstract}

\maketitle

\section*{Introduction}

\textbf{A. General setup.}
The \emph{group} $F$ introduced by Richard Thompson in 1965 is the group of the orientation preserving piecewise linear dyadic self-homeo\-mor\-phisms of the closed unit interval $[0,1]$. Arguably, the most important open question about Thompson's group $F$ is the one about its \emph{amenability} (see Cannon -- Floyd \cite{Cannon-Floyd11} for its history). Indeed, due to the plethora of rather unusual properties of this group (described, for instance, by Cannon -- Floyd -- Parry in \cite{Cannon-Floyd-Parry96}) either answer would imply very interesting consequences. This problem has recently attracted a lot of attention, with an impressive number of failed attempts to prove either amenability or non-amenability of the group~$F$.

We are inclined to believe that Thompson's group $F$ is non-amenable, and several recent papers provide circumstantial evidence for that. There are various computational experiments due to Burillo -- Cleary -- \linebreak Wiest \cite{Burillo-Cleary-Wiest07}, Haagerup -- Haagerup -- Ramirez Solano \cite{Haagerup-Haagerup-RamirezSolano15}, Elder -- Rechnitzer -- van Rensburg \cite{Elder-Rechnitzer-vanRensburg15}. Besides that, Moore \cite{Moore13} proved that if the group $F$ is amenable, then its F{\o}lner function must be growing very fast (note, however, that, as it was shown by Erschler \cite{Erschler06a}, even the groups of intermediate growth may have the F{\o}lner function growing faster than any given function).

The solution of a conjecture of Furstenberg \cite{Furstenberg73} by Rosenblatt \cite{Rosenblatt81} and by Kaimanovich -- Vershik \cite{Kaimanovich-Vershik83} provides the following characterization of amenability in terms of the \emph{Liouville property}. A~countable group is amenable if and only if it carries a non-degenerate \emph{Liouville random walk} (i.e., such that its \emph{Poisson boundary} is trivial). Therefore, a possible approach to proving non-amenability of a given group consists in showing that there are no non-degenerate Liouville random walks on it. Here we make the first step in this direction for Thompson's group $F$ by showing that \emph{the random walk $(F,\mu)$ is non-Liouville for any strictly non-degenerate finitely supported measure $\mu$}. In particular, \emph{the simple random walk on the Cayley graph of $F$ determined by any finite generating set is non-Liouville.}

Independently of the amenability issue, Thompson's group $F$ was among the first examples of finitely generated groups with an intermediate (between 0 and 1) value of the \emph{Hilbert space compression} which were exhibited by Arzhantseva -- Guba -- Sapir \cite{Arzhantseva-Guba-Sapir06}. They proved that for the group $F$ this value is equal to $1/2$ and asked about existence of a compression function strictly better than the square root \cite[Question~1.4]{Arzhantseva-Guba-Sapir06}. However, as it has been pointed out by Gournay \cite{Gournay14}, in view of the results of Naor -- Peres \cite{Naor-Peres08}, if such a function exists, then the group must be Liouville. In combination with our result, it allowed Gournay to answer the above question of Arzhantseva -- Guba -- Sapir in the negative: the best Hilbertian equivariant compression function for Thompson's group $F$ is (up to constants) the square root function.

\medskip

\textbf{B. Main results.}
The overwhelming majority of the currently known examples of an explicit non-trivial behavior at infinity for random walks on discrete groups falls into one of the following two classes. The first class consists of the examples, for which this behaviour is due to some kind of \emph{boundary convergence} usually related to more or less pronounced manifestations of hyperbolicity (the most representative example being, of course, the \emph{word hyperbolic groups}), see the papers by Kaimanovich -- Vershik \cite{Kaimanovich-Vershik83}, Kaimanovich \cite{Kaimanovich85, Kaimanovich89, Kaimanovich94a, Kaimanovich00a}, Kaimanovich -- Masur \cite{Kaimanovich-Masur96}, Karlsson -- Margulis  \cite{Karlsson-Margulis99}, Karlsson -- Ledrappier \cite{Karlsson-Ledrappier07}, Maher -- Tiozzo \cite{Maher-Tiozzo14p}. For the examples from the second class a non-trivial behavior at infinity is provided by \emph{pointwise stabilization of random configurations} in a certain way associated with the sample paths of the random walk (this situation is exemplified by the \emph{lamplighter groups}), see Kaimanovich -- Vershik \cite{Kaimanovich-Vershik83}, Kaimanovich \cite{Kaimanovich83a, Kaimanovich91}, Erschler \cite{Erschler04a, Erschler04, Erschler11a}, Karlsson --Woess \cite{Karlsson-Woess07}, Sava \cite{Sava10}, Lyons -- Peres \cite{Lyons-Peres15}, Juschenko -- Matte Bon -- Monod -- de la Salle \cite{Juschenko-MatteBon-Monod-delaSalle15p}.

In the case of Thompson's group $F$, true to its ambivalent nature, we actually construct \emph{asymptotic behaviours \mbox{($\mu$-boundaries)} of both types}. Note that although the descriptions of these $\mu$-boundaries are quite different, currently we do not know anything about their mutual position in the lattice of all $\mu$-boundaries determined by a fixed step distribution $\mu$. In particular, \emph{a priori} it is not excluded that both these $\mu$-boundaries might in fact coincide with the full Poisson boundary.

Our main tool is the ``canonical'' \emph{Schreier graph} $\G$ of the group~$F$ (endowed with the standard generators) on the dyadic-rational orbit of its canonical action on the unit interval. In principle one could argue just in terms of this action. However, it turns out to be much more convenient to ``change the coordinates'' and to realize Thompson's group $F$ as a subgroup $\wt F\cong F$ of the \emph{group $PLF(\R)$ of piecewise linear homeomorphisms of the real line}. [Actually, the geometry of the graph $\G$ completely described by Savchuk \cite{Savchuk10} is really begging for this coordinate change.] Then one of the generators of $F$ becomes just the translation $\g\mapsto\g-1$ on $\R$. The dyadic-rational orbit in the unit interval becomes the dyadic-rational orbit in $\R$, so that we can identify the vertex set of $\G$ with the dyadic-rational line $\Zd\subset\R$.

As we have already said, we construct two non-trivial $\mu$-boundaries of the group $F$. The first one is inspired by an analogy between Thompson's group $F\cong\wt F$ and the \emph{lamplighter groups}. Let $\fun(\G,\Z)$ (respectively, $\Fun(\G,\Z)$) denote the additive group of finitely supported (respectively, of all) \emph{$\Z$-valued configurations on the graph} $\G$. We assign to any element $g$ of the group $\wt F\subset PLF(\R)$ a configuration $\C_g\in\fun(\G,\Z)$ on $\G\cong\Zd$. The value of $C_g$ at a point $\g$ is the \emph{logarithmic increment of the slope} of $g$ at~$\g$ (so that the support of~$\C_g$ is precisely the discontinuity set of the derivative of $g$). Or, in a somewhat different terminology, the support of the configuration $\C_g$ is the set of the \emph{break points} of $g$, and its values are the logarithms of the \emph{jumps} of $g$.

By the chain rule the sequence of the configurations $\C_{g_n}$ along a sample path $(g_n)$ of the random walk $(\wt F,\mu)$ satisfies a simple recursive relation. This relation is completely analogous to that for the random walks on the lamplighter groups. In precisely the same way as with the lamplighter groups \cite{Kaimanovich-Vershik83,Kaimanovich83a}, if the measure $\mu$ is finitely supported, then the transience of the induced random walk on~$\G$ implies that the sequence $\C_{g_n}$ almost surely converges to a random limit configuration $\C_\infty\in\Fun(\G,\Z)$. Then it is not hard to verify that the limit configuration $\C_\infty$ can not be the same for all sample paths. Therefore, the space $\Fun(\G,\Z)$ endowed with the arising hitting distribution is a non-trivial $\mu$-boundary.

This argument (once again, precisely in the same way as for the lamplighter groups) is hinged on the \emph{transience} of the induced random walk on $\G$. We establish it by showing that for the simple random walk on $\G\cong\Zd$ there is a ``drift'' which forces the 2-adic norm to go to infinity along the sample paths (this is the original argument we referred to in \cite{Kaimanovich04p}). The classical comparison technique then leads to the transience of all random walks on $\G$ driven by strictly non-degenerate probability measures $\mu$ on $F\cong\wt F$.

Alternatively, the transience of the simple random walk on $\G$ can be directly obtained from the explicit description of the geometry of the graph $\G$ due to Savchuk \cite{Savchuk10}. Indeed, 
after removing from $\G$ a countable set of rays isomorphic to $\Z_+$ the remaining \emph{skeleton} $\ov\G$ is a tree roughly isometric to the standard binary tree. This fact readily implies that $\G$ is transient.

In a recent preprint \cite{Mishchenko15} Mishchenko gives yet another proof of the transience of the Schreier graph $\G$, which is based on an explicit Dirichlet norm estimate. He then notices that due to a specific geometry of $\G$ this transience implies existence of a non-trivial behaviour at infinity for the simple random walk on $\G$. Since it is induced by the simple random walk on Thompson's group $F$, the latter also has a non-trivial behaviour at inifnity. This argument gives a different proof of the absence of the Liouville property for the group $F$.

Mishchenko settles just for showing non-triviality of the Poisson boundary by exhibiting a non-trivial finite partition of it. However, his observation can be developed much further. Indeed, a transient finite range random walk on any graph necessarily converges in the end compactification. If the graph is endowed with a transitive group action which commutes with the transition operator of the random walk, then the dependence of this limit end on the starting point is equivariant, so that it is the same boundary behaviour of the random walk on the group that is exhibited independently of the starting point in the graph. However, this is not the case in our situation, and here, in order to obtain a $\mu$-boundary (i.e., an \emph{equivariant} quotient of the Poisson boundary), one has to consider the collections of the limit ends of the induced random walk on $\G$ parameterized by \emph{all} possible starting points. In other words, the resulting $\mu$-boundary is realized on the \emph{space of sections of the end bundle}
$$
\G\times\p\G\to\G \;, \qquad (\g,\w)\to\g
$$
over the graph $\G$ (here $\p\G$ is the space of ends of $\G$). One can consider these sections as \emph{``end fields''} on $\G$, or, in yet another termonilogy, as elements of the \emph{space $\Fun(\G,\p\G)\cong(\p\G)^\G$ of $\p\G$-valued configurations} on $\G$. By using certain self-similar features of the graph $\G$ (although its group of automorphisms is trivial, it has a lot of \emph{partial isomorphisms} between its subsets), it is easy to check that this $\mu$-boundary is non-trivial.

Our construction of a $\mu$-boundary of the group $F$ on the configuration space $\Fun(\G,\Z)$
heavily used the specific features of this group. On the contrary, the above construction of a $\mu$-boundary on the space of sections $\Fun(\G,\p\G)$ of the end bundle over the Schreier graph $\G$ is much more general. It produces a non-trivial $\mu$ boundary for a strictly non-degenerate finitely supported step distribution $\mu$ on an \emph{arbitrary} finitely generated group $G$ whenever there exists a Schreier graph $\G$ of this group with the following two properties:
\begin{itemize}
\item[(i)]
the graph $\G$ is transient with respect to the simple random walk;
\item[(ii)]
for the induced random walk $(\G,\mu)$ its harmonic (hitting) distributions on the space of ends $\p\G$ are not one-point measures.
\end{itemize}

Returning to Thompson's group $F\cong\wt F$, let us notice that the space of ends of its canonical Schreier graph $\G$ splits into a disjoint union
$$
\p\G = \p\ov\G \cup \p_-\G \cup \p_+\G \;.
$$
Here $\p\ov\G$ is the space of ends of the skeleton $\ov\G$, which can be identified with the space $\Z_2^{(1)}=\Z_2\setminus 2\Z_2$ of 2-adic integers of norm 1. The sets $\p_-\G,\p_+\G$ are countably infinite and correspond to the rays attached to the respective subsets of the skeleton. The components of the above decomposition, on which the hitting measures are actually concentrated, can be further described in terms of the barycentres of the images of the measure $\mu$ under its homomorphisms to $\Z$. In particular, in the centered case the hitting measures are supported by $\Fun(\G,\p\ov\G)$.

\medskip

\textbf{C. Structure of the paper.}
The paper has the following structure. In \secref{sec:prelim} we remind the background definitions concerning the Poisson boundary of a random walk on a discrete group, the Liouville property and its links with amenability. We also remind how one obtains a non-trivial behaviour at infinity for the lamplighter groups, as this was our source of inspiration for constructing one of the $\mu$-boundaries of Thompson's group (the one realized on the configuration space $\Fun(\G,\Z)$).

Further, in \secref{s:f} we introduce Thompson's group $F$ and describe its realization as a subgroup $\wt F\cong F$ of the group $PLF(\R)$ of piecewise linear homeomorphisms of the real line. 

In \secref{s:sch} we establish the transience of random walks on the Schreier graph~$\G$. First we do that for the simple random walk (\thmref{thm:transimple}), and then, by using a general comparison criterion (\thmref{thm:r}), for the random walk determined by any strictly non-degenerate step distribution (\thmref{thm:transall}).

In \secref{s:config} we construct a non-trivial $\mu$-boundary of the random walks on Thompson's group $F\cong\wt F$, which is determined by stabilizing $\Z$-valued configurations on the Schreier graph $\G$. It is realized on the configuration space $\Fun(\G,\Z)$ (\thmref{thm:conv}).

In \secref{s:geom} we establish several geometric properties of the Schreier graph $\G$ and its space of ends.

These properties are then used in \secref{s:ends} in order to construct another non-trivial $\mu$-boundary of random walks on Thompson's group $F\cong\wt F$. It is determined by the convergence of the induced random walk on the Schreier graph $\G$ to its ends, and it is realized on the space of sections $\Fun(\G,\p\G)$ of the end bundle over $\G$ (\thmref{thm:conv2}). In \thmref{thm:gen} we generalize this result to arbitrary groups which admit a transient Schreier graph with a non-trivial space of ends. In \thmref{thm:conv2a} we further specify the components of the space $\Fun(\G,\p\G)$, which actually support the $\mu$-boundary constructed in \thmref{thm:conv2}.

Finally, \secref{s:concl} contains a discussion of possible future developments and ramifications.

\medskip

\textbf{D. Historical comments.}
The main result of this paper (absence of the Liouville property for Thompson's group $F$) was presented in the talks ``Boundary behaviour of the Thompson group'' at the conference ``Combinatorial, Geometric, and Dynamical Aspects of Infinite Groups'' (1-6 June 2003, Gaeta, Italy) and at the special session ``Probabilistic and Asymptotic Aspects of Group Theory'' of the AMS meeting in Athens, Ohio (26-27 March 2004), and the slides of these talks \cite{Kaimanovich04p} were circulated at that time. However, it took quite a while to complete the present version, and in the meantime some of its ideas and tools were independently developed by other authors. The geometry of the Schreier graphs of the group $F$ corresponding to its canonical action on the unit interval was completely described by Savchuk \cite{Savchuk10,Savchuk15}; Mishchenko \cite{Mishchenko15} published a new proof of the absence of the Liouville property for the group $F$ based on the analysis of Savchuk.

\medskip

\textbf{E. Acknowledements.}
I would like to thank the editors of the present volume Tullio Cecche\-ri\-ni-Silberstein, Maura Salvatori and Ecaterina Sava-Huss for their infinite patience, understanding and co\-ope\-ration during the preparation of this article. I am grateful to the anonymous referee for a very interesting and competent report. Anna Erschler and Yair Hartman made a number of valuable comments and suggestions. Finally, last but not least, my thanks go to Wolfgang Woess, who is at the origin of this volume, and with whose friendship and wisdom I have been honoured for many years.

\section{Preliminaries} \label{sec:prelim}

\subsection{Random walks on groups} \label{sec:rw}

The (right) \textsf{random walk} $(G,\mu)$ on a \textsf{countable group} $G$ determined by a probability measure (the \textsf{step distribution}) $\mu$ on $G$ is the Markov chain on $G$ with the \textsf{transitions}
$$
g \mapstoto^{\mu(h)} gh \;, \qquad g\in G \;,
$$
i.e., its \textsf{transition probabilities} $\pi_g$ are the translates
$$
\pi_g = g\mu \;, \qquad g\in G \;,
$$
of the measure $\mu$. Any \textsf{initial distribution} $\th$ on $G$ determines the associated \textsf{Markov measure} $\P_\th$ on the \textsf{space $G^\ZZ$} of \textsf{sample paths} $\gg=(g_0,g_1,\dots)$.

By $\P=\P_{\d_e}$ we denote the probability measure on the path space whose initial distribution is the \textsf{point measure} $\d_e$ concentrated on the \textsf{identity} $e$ of the group $G$. For the random walk issued from the identity of the group (so that $g_0=e$), its position $g_n$ at time $n$ is the product $g_n=h_1 h_2\dots h_n$ of $n$ independent identically $\mu$-distributed \textsf{increments}~$h_i$, and the distribution of $g_n$ (the $n$-th marginal distribution of the measure $\P$) is the $n$-fold \emph{convolution} of the measure~$\mu$.

\subsection{Poisson boundary}

The measure $\P_\#$, whose initial distribution is the \textsf{counting measure} $\#$ on the group $G$, is invariant with respect to the \textsf{time shift} $\T$ on the path space. The \textsf{Poisson boundary} $\p_\mu G$ of the random walk $(G,\mu)$ is the \emph{space of ergodic components} of the time shift~$\T$ on the space $(G^{\Z_+},\P_\#)$, and any initial probability distribution~$\th$ on $G$ determines the corresponding \textsf{harmonic measure} $\nu_\th$ on $\p_\mu G$ as the image of the measure $\P_\th\prec\P_\#$ on the path space under the \textsf{quotient map} $G^{\Z_+}\to\p_\mu G$. We emphasize that \emph{the Poisson boundary is defined in the measure category only}. Below, when talking about the Poisson boundary $\p_\mu G$, we shall always (unless otherwise specified) endow it with the \textsf{measure $\nu=\nu_{\d_e}$}, which is the harmonic measure of the initial distribution $\d_e$.

The Poisson boundary is equipped with the natural (left) \textsf{action} of the group $G$ induced by the coordinate-wise left translations on the path space, and for an arbitrary initial distribution~$\th$
$$
\nu_\th = \th * \nu = \sum_g \th(g) \, g\nu \;.
$$
Although the harmonic measure $\nu$ need not be quasi-invariant (if the \textsf{semigroup $\sgr\mu$} generated by the \textsf{support} $\supp\mu$ of the measure $\mu$ is smaller than $G$), it is $\mu$-\textsf{stationary} with respect to this action in the sense that
$$
\mu * \nu = \sum_g \mu(g) \, g\nu = \nu \;.
$$

\subsection{Poisson formula}

A function $f$ on $G$ is called \textsf{$\mu$-harmonic} if it is preserved by the \textsf{Markov operator of the random walk} $(G,\mu)$
$$
P_\mu f(g) = \sum_h \mu(h) f(gh) \;,
$$
i.e., if $P_\mu f = f$. The \textsf{Poisson formula}
\begin{equation} \label{eq:P}
f(g) = \lan \wh f, g\nu \ran \;, \qquad g\in G \;,
\end{equation}
establishes an isometric isomorphism between the \textsf{space} $H^\infty(G,\mu)$ of bounded $\mu$-harmonic functions $f$ on the group $G$ and the $L^\infty$ space on the Poisson boundary $\p_\mu$ with respect to the quotient measure class determined by the measure $\P_\#$. It consists in assigning to a boundary function $\wh f$ the function $f$ on $G$ whose value at a point $g\in G$ is the result of the integration of the function $\wh f$ against the translate $g\nu$ of the harmonic measure $\nu$.

If $g$ in the Poisson formula \eqref{eq:P} is only allowed to take values in the semigroup $\sgr\mu$, then \eqref{eq:P} becomes an isometric isomorphism between the space of bounded $\mu$-harmonic functions on $\sgr\mu$ and $L^\infty(\p_\mu G,\nu)$. This property uniquely characterizes the Poisson boundary. Namely, any $G$-space $(B,\la)$, for which formula \eqref{eq:P} is an isometric isomorphism between the spaces $H^\infty(\sgr\mu,\mu)$ and $L^\infty(B,\la)$, is isomorphic to the Poisson boundary $(\p_\mu G,\nu)$ in the category of measure $G$-spaces. See Kaimanovich -- Vershik \cite{Kaimanovich-Vershik83}, Kaimanovich \cite{Kaimanovich92, Kaimanovich00a}, Furman \cite{Furman02}, Erschler \cite{Erschler10} for more background on the Poisson boundary.

\subsection{Stability of the Liouville property}

A random walk $(G,\mu)$ is called \textsf{Liouville} if the harmonic measure $\nu$ is a point measure. The reason for this terminology is that in view of the Poisson formula \eqref{eq:P} this property means that there are no non-constant bounded $\mu$-harmonic functions on the semigroup $\sgr\mu$. One can easily see that the latter property is actually equivalent to the absence of non-constant bounded $\mu$-harmonic functions on the whole \textsf{group} $\gr\mu$ generated by $\supp\mu$ as well.

There are numerous examples showing that one can have both Liouville and non-Liouville random walks on the same group (e.g., see Kaimanovich -- Vershik \cite{Kaimanovich-Vershik83}, Kaimanovich \cite{Kaimanovich83a}, Erschler \cite{Erschler04, Erschler04a}, Bartholdi -- Erschler \cite{Bartholdi-Erschler11}), and it is still not clear how the Liouville property for random walks on the same group depends on the step distribution $\mu$. 

\vfill\eject

The main open problem here is
\begin{quote}
\emph{whether all finitely supported symmetric measures $\mu$ with $\gr\mu=G$ on the same finitely generated group $G$ are either Liouville or non-Liouville simultaneously},
\end{quote}
or, in a somewhat different form,
\begin{quote}
\emph{whether the Liouville property is stable with respect to rough isometries}.
\end{quote}

In the absence of group invariance, for general graphs or Riemannian manifolds, the corresponding couterexamples to the stability of the Liouville property were first constructed by Terry Lyons \cite{Lyons87}. We believe that such counterexamples should exist in the group setup as well. However, the pool of groups, for which the Liouville property has been studied, still remains quite limited in spite of some significant recent progress, see Bartholdi -- Virag \cite{Bartholdi-Virag05}, Brieussel \cite{Brieussel09}, Bartholdi -- Kaimanovich -- Nekrashevych \cite{Bartholdi-Kaimanovich-Nekrashevych10}, Amir -- Angel -- Vir\'ag \cite{Amir-Angel-Virag13}, Amir -- Angel -- Matte Bon -- Vir\'ag \cite{Amir-Angel-MatteBon-Virag13p}, Matte Bon \cite{MatteBon14}, Kotowski -- Vir\'ag \cite{Kotowski-Virag15}, Juschenko -- Matte Bon -- Monod -- de la Salle \cite{Juschenko-MatteBon-Monod-delaSalle15p}.

\subsection{Liouville groups}

Because of the above there are different ways of calling a group $G$ \textsf{Liouville} depending on which class of step distributions $\mu$ on $G$ one considers (e.g., all measures, just finitely supported measures, finitely supported symmetric measures, etc.). According to one popular definition a group $G$ is called Liouville if the random walk $(G,\mu)$ is Liouville for any (possibly degenerate) finitely supported symmetric probability measure $\mu$ on $G$, see Matte Bon \cite{MatteBon14}. Sometimes one also talks about the Liouville property of a group $G$ with respect to a finite symmetric generating set $K$ (in which case one takes for $\mu$ the measure equidistributed on $K$), see Gournay \cite{Gournay14}.

\subsection{Amenability and the Liouville property}

Liouville groups (no matter which definition one takes) are always amenable, which goes back to Furstenberg \cite{Furstenberg73}. The reason for that is that a random walk $(G,\mu)$ is Liouville if and only if the sequence of Cesaro averages of the convolution powers of $\mu$ is asymptotically invariant with respect to the left action of the group $\gr\mu$ on itself, see Kai\-ma\-no\-vich -- Vershik \cite{Kaimanovich-Vershik83}. Therefore, the Liouville property of the random walk $(G,\mu)$ implies amenability of the group $\gr\mu$ in a very constructive way. Conversely, as it was first conjectured by Furstenberg in \cite{Furstenberg73}, and proved by Rosenblatt \cite{Rosenblatt81} and Kaimanovich -- Vershik \cite{Kaimanovich-Vershik83}, any amenable group $G$ carries a symmetric measure $\mu$ with $\supp\mu=G$ such that the random walk $(G,\mu)$ is Liouville. Note that there are amenable groups $G$, for which such a measure $\mu$ can not be chosen finitely supported, as it was shown by Kaimanovich \cite{Kaimanovich83a}, or even to have finite entropy, as it was shown by Erschler \cite{Erschler04}.

Therefore, a possible approach to proving non-amenability of a given group $G$ consists in showing that it carries no Liouville random walks with \textsf{non-degenerate} step distributions $\mu$ (i.e., such that $\gr\mu=G$). Actually, in view of the above result of Rosenblatt and Kai\-ma\-no\-vich -- \linebreak Vershik, it is enough to consider just the measures with $\sgr\mu=G$ (we shall call such measures \textsf{strictly non-degenerate}; sometimes they are also called \textsf{adapted}), or even just with $\supp\mu=G$. Here me make the first step in this direction for Thompson's group $F$ by showing that random walks on $F$ with finitely supported strictly non-degenerate step distributions are non-Liouville.

\subsection{Lamplighter groups} \label{sec:lamp}

Our approach to the Liouville property for the group $F$ consists in using the transience of the induced random walk on an auxiliary homogeneous space $\G$ of the group to prove that certain configurations on $\G$ associated with the elements of $F$ stabilize along the sample paths of the random walk. It is precisely this idea that was first used for exhibiting a non-trivial boundary behavior for random walks on \emph{lamplighter groups}, see Kaimanovich -- Vershik \cite{Kaimanovich-Vershik83} and Kaimanovich \cite{Kaimanovich83a}, and, to begin with, we shall outline it in the lamplighter context.

\begin{dfn}
The \textsf{lamplighter group} $\L(G)$ \textsf{over a base group} $G$ is the semi-direct product of the group $G$ and the additive group $\fun(G,\Z/2\Z)$ of finitely supported $\Z/2\Z$-valued configurations on $G$ endowed with the action of $G$ by translations.
\end{dfn}

\begin{rem}
In the context of functional and stochastic analysis these groups were first introduced by Vershik and the author \cite{Kaimanovich-Vershik83} under the name of the \emph{groups of dynamical configurations}. However, this term did not stick, and the current generally accepted standard is to call them \emph{lamplighter groups}. In the algebraic language $\L(G)$ is the \emph{wreath product} with the \emph{active group} $G$ and the \emph{passive group} $\Z/2\Z$. Since our purpose is just to outline the general idea, we are not discussing here more general wreath products (for which see, for instance, Kaimanovich \cite{Kaimanovich91} and Erschler \cite{Erschler04a, Erschler06}).
\end{rem}

As a set, $\L(G)$ is the usual product of $G$ and $\fun(G,\Z/2\Z)$, so that
$$
\L(G) = \{(g,\F): g\in G,\;\F\in\fun(G,\Z/2\Z) \} \;.
$$
However, the group operation in $\L(G)$ is ``skewed'' by using the left action
of $G$ on $\fun(G,\Z/2\Z)$ by the group automorphisms
\begin{equation} \label{eq:act0}
\mathbf{T}^g \F (h) = \F(g^{-1} h)
\end{equation}
induced by the left action of the group $G$ on itself by translations, so that the group multiplication in $\L(G)$ is
\begin{equation} \label{eq:mult}
(g_1,\F_1) (g_2,\F_2) = ( g_1 g_2, \F_1 + \mathbf{T}^{g_1}\F_2 ) \;.
\end{equation}
The identity of $\L(G)$ is $(e,\vn)$, where $e$ is the identity of $G$, and $\vn$ (the \textsf{empty configuration} defined as $\vn(h)=0$ for any $h\in G$) is the identity of $\fun(G,\Z/2\Z)$.

\subsection{Stabilization of configurations} \label{sec:stablamp}

Given a probability measure~$\mu$ on the group $\L(G)$, the positions $(g_n,\F_n)$ of the corresponding random walk at two consecutive time moments are related as
$$
( g_{n+1},\F_{n+1} ) = ( g_n,\F_n ) ( h_{n+1},\f_{n+1} ) \;,
$$
where $(h_i,\f_i)$ are the increments of the random walk. Therefore, by formula \eqref{eq:mult},
\begin{equation} \label{eq:llproduct}
\begin{cases}
g_{n+1} = g_n h_{n+1} \;, \\
\F_{n+1} = \F_n + \mathbf{T}^{g_n} \f_{n+1} \;.
\end{cases}
\end{equation}
In particular, the first component $g_n\in G$, i.e., the image of the random walk $(\L(G),\mu)$ under the group homomorphism
$$
\L(G) \to G \;, \qquad (g,\F)\mapsto g \;,
$$
performs the \textsf{quotient random walk} on $G$ driven by the image $\mu'$ of the measure $\mu$ under the homomorphism $\L(G)\to G$.

If the quotient random walk $(G,\mu')$ on $G$ is \textsf{transient} (i.e., its sample paths almost surely go to infinity on $G$), and the support of the measure $\mu$ is finite (which implies that there is a finite set $A\subset G$ which contains the supports of all the increments $\f_n$), then formula \eqref{eq:llproduct} implies that the configurations $\F_n$ almost surely stabilize as $n\to\infty$. It means that there exists a pointwise random limit
$$
\F_\infty = \lim_{n\to\infty} \F_n = \f_1 + \mathbf{T}^{g_1}\f_2 + \mathbf{T}^{g_2}\f_3 + \dots \;,
$$
which belongs to the space $\Fun(G,\Z/2\Z)$ of all (not necesssrily finitely supported) $\Z/2\Z$-valued configurations on $G$.

It is easy to see that under natural conditions on the measure $\mu$ the limit configuration $\F_\infty$ can not be the same for a.e.\ sample path, and therefore the Poisson boundary of the random walk $(\L(G),\mu)$ is non-trivial. This is the case if the measure $\mu$ is non-degenerate in $\L(G)$, but actually this assumption can be siginificantly weakened, see Kaimanovich \cite{Kaimanovich91}. The map, which assigns to a sample path $(g_n,\F_n)$ the associated limit configuration $\F_\infty\in\Fun(G,\Z/2\Z)$, is obviously $\L(G)$-equivariant (with respect to the action on $\Fun(G,\Z/2\Z)$ determined by the ``configuration component'' of formula \eqref{eq:mult}). Therefore, 
the configuration space $\Fun(G,\Z/2\Z)$ endowed with the resulting hitting distribution is a \textsf{$\mu$-boundary} (an equivariant quotient of the Poisson boundary).

\vfill\eject

\section{Thompson's group $F$ as a subgroup of $PLF(\R)$} \label{s:f}

\subsection{The group $F$ and its generators} \label{sec:genF}

\begin{dfn} \label{dfn:thomp}
The \textsf{Thompson group} $F$ is the group of the orientation preserving piecewise linear  self-homeo\-mor\-phisms $g$ of the closed unit interval $[0,1]$ that are differentiable except for finitely many \textsf{break points}~$t_i$, which are dyadic rational numbers, and such that the \textsf{slopes}~$a_i$ are integer powers of 2:
$$
\begin{aligned}
g(t) = a_i t + b_i &\quad\text{on}\quad  [t_i,t_{i+1}]\;, \qquad \\
&t_i= \frac{k_i}{2^{m_i}} \in {\textstyle \Zd}\;,\quad a_i=2^{n_i}\;,\quad b_i\in{\textstyle \Zd} \;,
\end{aligned}
$$
see \figref{fig:graph}.
\end{dfn}

\begin{center}
\psfrag{x}[][]{$t_i$}
 \psfrag{y}[][]{$t_{i+1}$}
 \psfrag{b}[][]{$0$}
 \psfrag{e}[][]{$1$}
\includegraphics{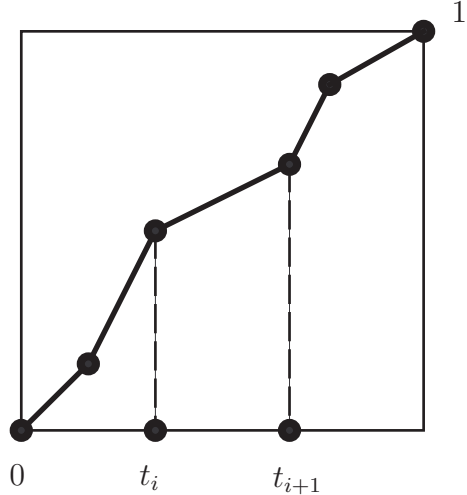}
\captionof{figure}{The graph of an element of Thompson's group $F$.} \label{fig:graph}
\end{center}

\bigskip

We refer to Cannon -- Floyd -- Parry \cite{Cannon-Floyd-Parry96} for the general background on the group $F$. In particular, it is finitely generated with the \textsf{generators}
$$
A(t) =
\begin{cases}
\frac{t}2   \;,  & 0\le t\le \frac12 \\
t-\frac14 \;,  & \frac12 \le t \le \frac34 \\
2t-1  \;,  & \frac34 \le t \le 1
\end{cases}
$$
and
$$
B(t) =
\begin{cases}
t   \;,  & 0\le t\le \frac12 \\
\frac{t}2+\frac14 \;,  & \frac12 \le t \le \frac34 \\
t-\frac18  \;,  & \frac34 \le t \le \frac78 \\
2t-1 \;, & \frac78 \le t \le 1
\end{cases}  \qquad,
$$
see \figref{fig:gen} (where the arrow indicates that the graph of $B$ restricted to the square $[\frac12,1]\times [\frac12,1]$ is precisely the graph of $A$ rescaled by a factor of 2).
\begin{figure}[h]
     \psfrag{o}[][l]{$0$}
     \psfrag{a}[][l]{$\frac14$}
     \psfrag{b}[][l]{$\frac12$}
     \psfrag{f}[][l]{$\frac58$}
     \psfrag{c}[][l]{$\frac34$}
     \psfrag{e}[][l]{$\frac78$}
     \psfrag{i}[][l]{$1$}
     \psfrag{A}[][l]{$A$}
     \psfrag{B}[][l]{$B$}
          \includegraphics{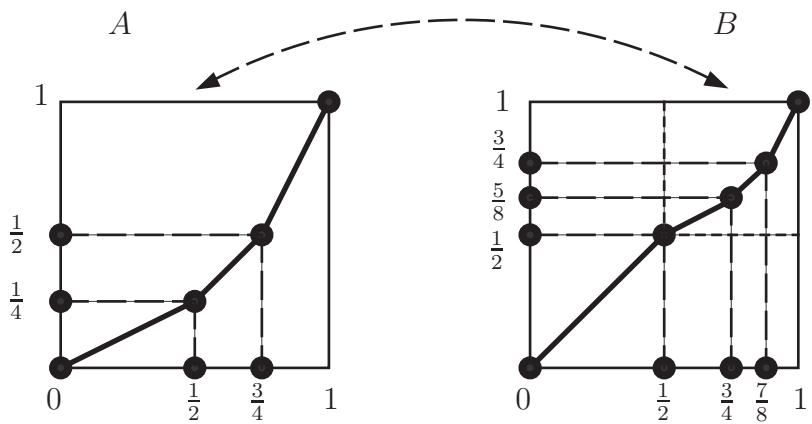}
\caption{The standard generators of Thompson's group $F$.} \label{fig:gen}
\end{figure}

The group $F$ is finitely presented, and it is determined by the relators 
$$
[AB^{-1},A^{-1}BA] \;,\qquad [AB^{-1},A^{-2}BA^2] \;,
$$
where $[x,y]=xyx^{-1}y^{-1}$ denotes the usual group theory commutator. The group $F$ is also often described by the infinite presentation
$$
\lan x_0,x_1,x_2,\dots | x_k^{-1} x_n x_k = x_{n+1} \;\text{for all}\;k<n \ran \;,
$$
which can be obtained from the previous one by putting $x_0=A$ and $x_n=A^{1-n}BA^{n-1}$ for $n>0$.

\subsection{Homomorphisms to $\Z$} \label{sec:f}

The abelianization $\Z^2$ of the group $F$ is freely generated by the images of the above generators $A$ and $B$. By $\chi_a$ and $\chi_b$ we denote the corresponding \textsf{homomorphisms of $F$ to $\Z$}, i.e., the homomorphisms $\chi_a$ and $\chi_b$ are defined by putting, respectively,
$$
\chi_a(A)=1 \;, \qquad \chi_a(B)=0 \;,
$$
and
$$
\chi_b(A)=0 \;, \qquad \chi_b(B)=1 \;.
$$
Geometrically, $-\chi_a(g)$ (respectively, $\chi_a(g)+\chi_b(g)$) is the base 2 logarithm of the slope of the graph of $g$ at the endpoint 0 (respectively, at the endpoint 1).

\subsection{Change of variables} \label{sec:change}

The group $F$ can also be realized as a subgroup of the \textsf{group of piecewise linear homeomorphisms of the real line} $PLF(\R)$ introduced by Brin -- Squier \cite{Brin-Squier85} (also see Haagerup -- Picioroaga \cite[Remark 2.5]{Haagerup-Picioroaga11}). 

More precisely, let $\t:[0,1]\to\R$ be defined by putting
\begin{equation} \label{eq:t}
\begin{aligned}
\t\left( \frac1{2^n} \right) = n+1 \;, \qquad n\le -1 \\
\t\left( 1- \frac1{2^n} \right) = n-1 \;, \qquad n \ge 1
\end{aligned}
\end{equation}
and by linear interpolation in between, see \figref{fig:change}.

\bigskip

\begin{center}
\psfrag{o1}[b][]{$0$}
 \psfrag{o7}[b][]{$1$}
 \psfrag{o4}[b][]{$\frac12$}
 \psfrag{o3}[b][]{$\frac14$}
 \psfrag{o2}[b][]{$\frac18$}
 \psfrag{o5}[b][]{$\frac34$}
 \psfrag{o6}[b][]{$\frac78$}
 \psfrag{a1}[][]{$-2$}
 \psfrag{a2}[][]{$-1$}
 \psfrag{a3}[][]{$0$}
 \psfrag{a4}[][]{$+1$}
 \psfrag{a5}[][]{$+2$}
\includegraphics[scale=1.0]{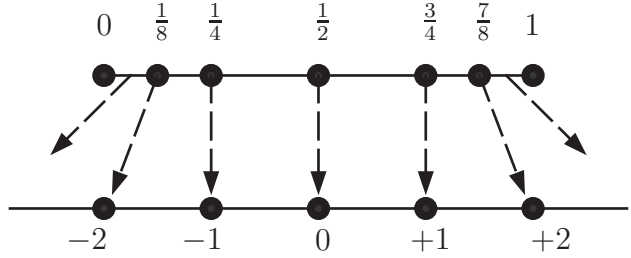}
\captionof{figure}{The change of variables from $[0,1]$ to $\R$.} \label{fig:change}
\end{center}

\bigskip

Then one can easily see that $\t$ is a bijection between the set of dyadic-rational numbers in the open unit interval $\Zd\cap (0,1)$ and the whole set of dyadic-rational numbers $\Zd$. After this \textsf{change of coordinates} the generators $A$ and $B$ of the group~$F$ (see \secref{sec:genF}) take the form
$$
\wt A(\g) = \g-1 \;, \qquad \g\in\R \;,
$$
and
$$
\wt B(\g) =
\begin{cases}
    \g \;, & \g \le 0, \\
    \frac{\g}2 \;, & 0\le \g \le 2, \\
    \g-1 \;, & \g \ge 2
\end{cases}
$$
Their inverses are, respectively,
$$
\wt A^{-1}(\g) = \g+1 \;, \qquad \g\in\R \;,
$$
and
$$
\wt B^{-1}(\g) =
\begin{cases}
    \g \;, & \g \le 0, \\
    2\g \;, & 0\le \g \le 1, \\
    \g+1 \;, & \g \ge 1 \;,
\end{cases}
$$
see \figref{fig:gen1}, where the graphs of the generators $\wt A$ and $\wt B$ themselves are drawn with solid lines, and the graphs of their inverses are drawn with dashed lines.

\begin{figure}[h]
     \psfrag{o}[][l]{$0$}
     \psfrag{i}[][l]{$1$}
     \psfrag{j}[][l]{$-1$}
     \psfrag{ii}[][l]{$2$}
     \psfrag{A}[][l]{$\wt A,\wt A^{-1}$}
     \psfrag{B}[][l]{$\wt B,\wt B^{-1}$}
          \includegraphics{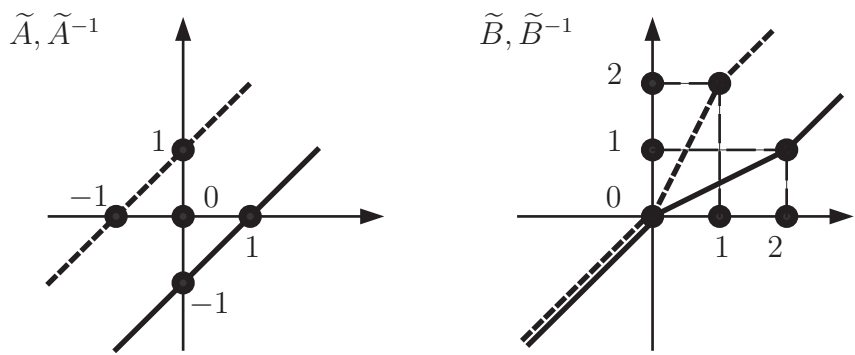}
\caption{The graphs of the standard generators of the group $\wt F$.} \label{fig:gen1}
\end{figure}

We shall denote \textsf{the image of the group $F$ under the change of coordinates} \eqref{eq:t} by $\wt F\cong F$ , i.e., $\wt F$ is the subgroup of $PLF(\R)$ generated by the transformations $\wt A$ and $\wt B$.

\subsection{Intrinsic description of $\wt F$} \label{sec:slope}

One can easily see that $\wt F$, as a subgroup of $PLF(\R)$, consists precisely of those transformations $g\in PLF(\R)$, for which the following three conditions are satisfied:

\begin{itemize}
\item[(i)]
The discontinuity points of the derivative $g'$ are all dyadic rational;
\item[(ii)]
The slopes of $g$ are integer powers of 2;
\item[(iii)]
For a sufficiently large $M>0$ the transformation $g$ has the form
$$
g(\g)=\g+C_- \quad \text{for} \quad \g\in (-\infty,-M) \;,
$$
and
$$
g(\g)=\g+C_+ \quad \text{for} \quad \g\in (M,\infty) \;,
$$
where the constants $C_-=C_-(g)$ and $C_+=C_+(g)$ are integers.
\end{itemize}

Clearly, the maps $C_\pm:\wt F\to\Z$ are group homomorphisms, and for the generators of the group $\wt F$
$$
C_-(\wt A)=C_+(\wt A)=-1 \;,
$$
whereas 
$$
C_-(\wt B)=0 \;, \qquad C_+(\wt B)=-1 \;,
$$
so that in terms of the homomorphisms $\chi_a,\chi_b:F\cong\wt F\to\Z$ introduced in \secref{sec:f}
$$
C_-=-\chi_a \quad \text{and} \quad C_+=-\chi_a-\chi_b \;.
$$

\section{Random walk on the Schreier graph $\G\cong\Zd$} \label{s:sch}

\subsection{The dyadic-rational Schreier graph}

For technical reasons (since traditionally one considers \emph{right} random walks on groups, for which the increments are added on the right), from now on we shall use the \emph{postfix notation} for the group $\wt F$. Thus, the \textsf{group operation} in~$\wt F$ will be defined as
\begin{equation} \label{eq:comp}
(g_1 g_2) (\g) = g_2 ( g_1 (\g)) \;,
\end{equation}
(rather than the more common $(g_1 g_2) (\g) = g_1 ( g_2 (\g))$ in the \emph{prefix notation} which we used in \secref{sec:genF}). Then $\wt F$ will naturally \textsf{act on $\R$ on the \emph{right}} as
\begin{equation} \label{eq:confact}
\g.g = g(\g) \;,\qquad \g\in\R,\; g\in\wt F \;.
\end{equation}
Our notation agrees with the one used when dealing with random walks on \emph{groupoids} (in particular, on the \emph{groupoids associated with group actions}), see Kaimanovich \cite{Kaimanovich05a}.

The set of \textsf{dyadic rational numbers} $\Zd\subset\R$ is obviously transitive for this action, so that $\Zd$ can be endowed with the structure of the \textsf{Schreier graph} $\G$ of the group $\wt F$ with respect to the \textsf{generating set} 
\begin{equation} \label{eq:k}
K=\{\wt A,\wt B,\wt A^{-1},\wt B^{-1}\} \;.
\end{equation}

\subsection{Simple random walk on the Schreier graph} \label{sec:srw}

The \textsf{simple random walk} on the graph $\G$ has the transitions
\begin{equation} \label{eq:srw}
\g \mapstoto^{\mu(g)} \g.g \;,
\end{equation}
where $\mu=\mu_K$ is the probability measure equidistributed on the generating set $K$, i.e.,
$$
\g \mapstoto
\begin{cases}
\displaystyle{\mapstoto^{\frac14}} \wt A(\g) \\
\displaystyle{\mapstoto^{\frac14}} \wt A^{-1}(\g) \\
\displaystyle{\mapstoto^{\frac14}} \wt B(\g) \\
\displaystyle{\mapstoto^{\frac14}} \wt B^{-1}(\g) \qquad .
\end{cases}
$$
More concretely, depending on the position of the point $\g$ these transitions are

$$
\begin{aligned}
\g &\mapstoto
\begin{cases}
\displaystyle{\mapstoto^{\frac12}} \g+1 \\
\displaystyle{\mapstoto^{\frac12}} \g-1
\end{cases}
\qquad,\quad \g\in(\infty,0] \cup [2,\infty) \;, \\
\g &\mapstoto
\begin{cases}
\displaystyle{\mapstoto^{\frac14}} \g+1 \\
\displaystyle{\mapstoto^{\frac14}} \g-1 \\
{\displaystyle\mapstoto^{\frac14}} \frac{\g}2 \\
\displaystyle{\mapstoto^{\frac14}} 2\g
\end{cases}
\qquad,\quad \g\in [0,1] \;, \\
\g &\mapstoto
\begin{cases}
\displaystyle{\mapstoto^{\frac12}} \g+1 \\
\displaystyle{\mapstoto^{\frac14}} \g-1 \\
{\displaystyle\mapstoto^{\frac14}} \frac{\g}2 \\
\end{cases}
\qquad,\quad \g\in [1,2] \;.
\end{aligned}
$$
We shall denote by $\Q$ the \textsf{probability measure on the path space of the simple random walk on $\G$ corresponding to the starting point~$0$}.

Let $|\g|_2$ be the usual \textsf{2-adic norm} of a number $\g\in\G\cong\Zd$, i.e.,
$$
|\g|_2 = 2^{-n} \;\quad \text{for}\; \g = (2m+1)2^n \;,
$$
and $|0|_2=0$.

\begin{thm} \label{thm:trans}
For $\Q$-a.e.\ sample path $(\g_n)$ of the simple random walk \eqref{eq:srw} on $\G$
$$
|\g_n|_2 \displaystyle{\toto_{n\to\infty}} \infty \;.
$$
\end{thm}

As a consequence we immediately obtain

\begin{thm} \label{thm:transimple}
The simple random walk on $\G$ is transient.
\end{thm}

\begin{proof}[Proof of \thmref{thm:trans}]
Let us first of all notice that, outside of the interval $(0,2)$ the transition probabilities of the random walk \eqref{eq:srw} are just
$$
p(\g,\g\pm 1)=\frac12 \;.
$$
Therefore, the interval $[0,2)$ is recurrent by the classical properties of the simple random walk on $\Z$. Moreover, starting from any point $\g\in [1,2)$ the random walk will eventually hit the interval $[0,1)$ (and the corresponding hitting distribution is supported by the points $\g-1$ and $\frac{\g}2$ with the equal weights $\frac12$). Thus, the interval $I=[0,1)$ is also recurrent. The transitions of the induced random walk on $I$ are, \linebreak for $\g\in \left[0,\frac12\right)$,
$$
\g \mapstoto
\begin{cases}
\displaystyle{\mapstoto^{\frac14}} \g \\
\displaystyle{\mapstoto^{\frac14}} \g+1 \mapstoto
\begin{cases}
\displaystyle{\mapstoto^{\frac12}} \g \\
{\displaystyle{\mapstoto^{\frac12}}} \frac{\g}2+\frac12
\end{cases}\\
{\displaystyle{\mapstoto^{\frac14}}} \frac{\g}2 \\
\displaystyle{\mapstoto^{\frac14}} 2\g \\
\end{cases} \qquad,
$$
i.e.,
$$
\g \mapstoto
\begin{cases}
\displaystyle{\mapstoto^{\frac38}} \g \\
{\displaystyle{\mapstoto^{\frac14}}} \frac{\g}2 \\
{\displaystyle{\mapstoto^{\frac18}}} \frac{\g}2+\frac12 \\
\displaystyle{\mapstoto^{\frac14}} 2\g \\
\end{cases} \;,
$$
and, for $\g\in \left[\frac12,1\right)$,
$$
\g \mapstoto
\begin{cases}
\displaystyle{\mapstoto^{\frac14}} \g \\
\displaystyle{\mapstoto^{\frac14}} \g+1
\mapstoto
\begin{cases}
\displaystyle{\mapstoto^{\frac12}} \g \\
{\displaystyle{\mapstoto^{\frac12}}} \frac{\g}2+\frac12
\end{cases}
\\
{\displaystyle{\mapstoto^{\frac14}}} \frac{\g}2 \\
\displaystyle{\mapstoto^{\frac14}} 2\g \mapstoto
\begin{cases}
{\displaystyle{\mapstoto^{\frac12}}} 2\g-1 \\
{\displaystyle{\mapstoto^{\frac12}}} \g
\end{cases}
\end{cases}
\qquad,
$$
i.e.,
$$
\g \mapstoto
\begin{cases}
\displaystyle{\mapstoto^{\frac12}} \g \\
{\displaystyle{\mapstoto^{\frac14}}} \frac{\g}2 \\
{\displaystyle{\mapstoto^{\frac18}}} \frac{\g}2+\frac12 \\
{\displaystyle{\mapstoto^{\frac18}}} 2\g-1 \\
\end{cases} \;.
$$
Therefore, for $\g\in I\setminus\left\{0,\frac12\right\}$ the logarithm $x=\log_2|\g|_2$ of the 2-adic norm changes as
$$
\begin{aligned}
x &\mapstoto
\begin{cases}
\displaystyle{\mapstoto^{\frac38}} x \\
\displaystyle{\mapstoto^{\frac38}} x+1 \\
\displaystyle{\mapstoto^{\frac14}} x-1 \\
\end{cases}
\qquad,\quad \g\in \left(0,\frac12\right) \;,
\\
x &\mapstoto
\begin{cases}
\displaystyle{\mapstoto^{\frac12}} x \\
\displaystyle{\mapstoto^{\frac38}} x+1 \\
\displaystyle{\mapstoto^{\frac18}} x-1 \\
\end{cases}
\qquad,\quad \g\in \left(\frac12,1\right) \;,
\end{aligned}
$$
i.e., these transitions on $\Z_+$ have a uniform positive drift, which implies the claim.
\end{proof}

\begin{rem}
It might be interesting to study the properties of the induced random walk on $I$ and of its stationary measure (which is most likely unique).
\end{rem}

\subsection{Comparison criterion for transience of Markov chains}

We shall now use a comparison argument in order to deduce the transience of general (not necessarily reversible!) random walks on the Schreier graph $\G$ from the transience just of the simple random walk on $\G$. Its idea goes back to Baldi -- Lohou\'e -- Peyri\`ere \cite{Baldi-Lohoue-Peyriere77}, and it has been quite popular ever since (e.g., see Varopoulos \cite{Varopoulos83}, Chen \cite{Chen91}, Woess \cite{Woess94} as well as the exposition in Woess' book \cite[Sections 2.C and 3.A]{Woess00}). For the sake of completeness we shall prove it here in the generality sufficient for our purposes by slightly modifying the arguments of Varopoulos from \cite[Section 4]{Varopoulos83}.

\begin{thm} \label{thm:r}
Let $P$ and $P'$ be two Markov operators on a countable state space $X$ with the respective transition probabilities $p(\cdot,\cdot)$ and $p'(\cdot,\cdot)$, and such that
\begin{itemize}
\item[(i)]
$P$ and $P'$ have a common stationary measure $m$;
\item[(ii)]
The operator $P$ is reversible with respect to the measure $m$, i.e., it is self-adjoint as an operator on the space $L^2(X,m)$, or, equivalently, 
$$
m(x)p(x,y)=m(y)p(y,x) \qquad\forall\,x,y\in X \;,
$$
\item[(iii)]
There exists $\e>0$ such that $P'\ge \e P$, i.e., 
$$
p'(x,y) \ge \e p(x,y) \qquad\forall\,x,y\in X \;.
$$
\end{itemize}
Then the transience of the operator $P$ implies the transience of the operator~$P'$.
\end{thm}

\begin{proof}
First of all, let us notice that the operator 
$$
Q= \frac{P-\e P'}{1-\e}
$$ 
is Markov and preserves the measure $m$. Therefore, for any $0<\la< 1$ and any $f\in L^2(X,m)$
$$
|\la \lan f, Qf \ran| \le |\lan f, Qf \ran| \le \| f \|^2 \;.
$$
It implies that
\begin{equation} \label{eq:neq}
\lan (I-\la P') f, f \ran \ge \e \| f \|^2_{P,\la} \ge 0 \;,
\end{equation}
where $\lan \cdot,\cdot \ran$ denotes the scalar product on the space $L^2(X,m)$,
\begin{equation} \label{eq:form}
\lan f,g \ran_{P,\la} = \lan f, (I-\la P) g \ran
\end{equation}
is the positive definite bilinear form on $L^2(X,m)$ determined by the operator $I-\la P$, and
$\|\cdot\|, \|\cdot\|_{P,\la}$ are the respective associated norms.

The Cauchy-Schwarz inequality for the form $\lan \cdot,\cdot \ran_{P,\la}$ then implies that for any $f\in L^2(X,m)$
$$
\begin{aligned}
\lan (I-\la P')^{-1} f, f \ran^2
&= \lan (I-\la P')^{-1} f, (I-\la P)(I-\la P)^{-1}f \ran^2 \\
&= \lan (I-\la P')^{-1} f, (I-\la P)^{-1}f \ran^2_{P,\la} \\
&\le \|(I-\la P')^{-1} f\|^2_{P,\la} \cdot \|(I-\la P)^{-1} f\|^2_{P,\la} \;.
\end{aligned}
$$
By \eqref{eq:neq},
$$
\|(I-\la P')^{-1} f\|^2_{P,\la} \le \frac1\e \lan (I-\la P')^{-1} f, f \ran \;,
$$
whereas
$$
\|(I-\la P)^{-1} f\|^2_{P,\la} = \lan (I-\la P)^{-1} f, f \ran \;,
$$
so that 
\begin{equation} \label{eq:neq2}
\lan (I-\la P')^{-1} f, f \ran^2 \le \frac1\e \lan (I-\la P')^{-1} f, f \ran \cdot \lan (I-\la P)^{-1} f, f \ran \;.
\end{equation}

Now, if $f=\d_o$ for a point $o\in X$, then 
$$
\Gr_{P,\la}(o,o) = \frac1{m(o)}\lan (I-\la P)^{-1} \d_o, \d_o \ran
$$ 
is the $\la$-Green kernel of the operator $P$ at the point $o$, and the same holds for the operator $P'$. Since $\la<1$,
$$
0 < \Gr_{P',\la}(o,o) < \infty \;,
$$
whence by \eqref{eq:neq2}
$$
\Gr_{P',\la}(o,o) \le \frac1\e \Gr_{P,\la}(o,o) \le \frac1\e \Gr_{P,1}(o,o) \;,
$$
which, by letting $\la\to 1$, implies the inequality
\begin{equation} \label{eq:neq3}
\Gr_{P',1}(o,o) \le \frac1\e \Gr_{P,1}(o,o) \;.
\end{equation}
Therefore, the finiteness of $\Gr_{P,1}(o,o)$ ($\equiv$ the transience of $P$) implies the finiteness of $\Gr_{P',1}(o,o)$ ($\equiv$ the transience of $P'$).
\end{proof}

\begin{rem}
We are not aware of any ``elementary'' proof of inequality~\eqref{eq:neq3}.
\end{rem}

\begin{rem}
The form \eqref{eq:form} decomposes as
$$
\lan \cdot,\cdot\ran_{P,\la} = \D_P + (1-\la) \lan \cdot,\cdot \ran \;,
$$
where $\D_P$ is the \emph{Dirichlet form} of the operator $P$.
\end{rem}

\subsection{General random walks}

Given a probability measure $\mu$ on $G$ we shall denote by $(\G,\mu)$ the \textsf{induced random walk} on $\G$
with the transitions
\begin{equation} \label {eq:rwg}
\g \mapstoto^{\mu(g)} \g.g \;.
\end{equation}
In other words, the measure $\Q_\g$ on the space of paths of the induced random walk issued from a point $\g\in\G$ is the image of the measure $\P$ on the path space of the random walk $(G,\mu)$ (see \secref{sec:rw}) under the map
\begin{equation} \label{eq:gg}
\gg = (g_n) \mapsto \g.\gg = (\g.g_n) \;.
\end{equation}

\begin{thm} \label{thm:transall}
Let $\mu$ be a strictly non-degenerate probability measure on the group $\wt F$. Then the induced random walk $(\G,\mu)$ is transient.
\end{thm}

\begin{proof}
This is a direct application of \thmref{thm:r}. By the strict non-degeneracy of the measure $\mu$, for any $g\in\wt F$ there exists $n=n(g)$ such that the $n$-fold convolution $\mu^{*n}(g)$ is strictly positive. Let
$$
n = \max \{ n(g): g\in K \} \;,
$$
and 
$$
\d = \min \{ \mu^{*n(g)}(g): g\in K \} > 0 \;,
$$
where $K$ is the generating set \eqref{eq:k}. Then the measure 
$$
\mu' = \frac1n \sum_{k=1}^n \mu^{*k}
$$
has the property that 
\begin{equation} \label{eq:qq}
\mu'(g) \ge \e \mu_K(g) \qquad\forall\,g\in K \;,
\end{equation}
with $\e = \d/n>0$.

Let us now take for $P$ (respectively, $P'$) the Markov operator of the simple random walk $(\G,\mu_K)$ (respectively, the operator of the random walk $(\G,\mu')$). If we take for $m$ the counting measure on $\G$, then conditions (i) and (ii) of \thmref{thm:r} are obviously satisfied, whereas condition (iii) follows from inequality \eqref{eq:qq}. Thus, \thmref{thm:transimple} implies transience of the random walk $(\G,\mu')$, and therefore of the random walk $(\G,\mu)$ as well.
\end{proof}

\begin{rem}
One should be able to significantly relax the condition $\sgr\mu=\wt F$ imposed on the measure $\mu$ in
\thmref{thm:transall}. We expect it to hold just under very mild assumptions on the group $\gr\mu$.
\end{rem}

\section{Non-trivial behaviour at infinity determined by stabilizing configurations} \label{s:config}

\subsection{Groups of configurations on $\G$}

Let $\fun(\G,\Z)$ (respectively, $\Fun(\G,\Z)$) denote the \textsf{additive group of finitely supported} (respectively, of all) \textsf{$\Z$-valued configurations on} the graph $\G$, cf. \secref{sec:lamp}. By using the right action \eqref{eq:confact} of the group $\wt F$ on $\G\cong\Zd$ we shall define the associated \emph{left} action of $\wt F$ on the configuration space $\Fun(\G,\Z)$ as
\begin{equation} \label{eq:act}
\mathbf{S}^g \f(\g) = \f(\g.g) = \f(g(\g)) \;,\qquad g\in\wt F,\;\g\in\G \;.
\end{equation}

\subsection{Configurations associated with the elements of $\wt F$}

Given a transformation $g\in\wt F$, we shall denote by $\C_g\in\fun(\G,\Z)$ the \textsf{associated configuration} on $\G$ defined as
$$
\C_g(\g) = \log_2 g'(\g+0) - \log_2 g'(\g-0) \qquad\forall\,\g\in\G \;.
$$
In other words, $\C_g(\g)$ is the difference between the base 2 logarithms of the left and the right slopes of $g$ at the point $\g$. The support of $\C_g$ is precisely the set of break points of $g$, i.e., the set of discontinuity points of the derivative $g'$, see \figref{fig:config}. 

Informally we shall say that the value $\C_g(\g)$ is the \textsf{logarithmic increment} of the slope of $g$ at the point $\g$. The ratio $\frac{g'(\g+0)}{g'(\g-0)}$ appears in the literature on Thompson's groups (e.g., see Liousse \cite{Liousse08}) under the name of the \textsf{jump} of $g$ at the break point $\g$, so that in these terms $\C_g(\g)$ is the logarithm of the jump of $g$ at $\g$.

\begin{center}
\psfrag{x}[][]{$\g_1$}
\psfrag{y}[][]{$\g_2$}
\psfrag{a}[][]{$\C(\g_1)=-2$}
\psfrag{b}[][]{$\C(\g_2)=1$}
\includegraphics{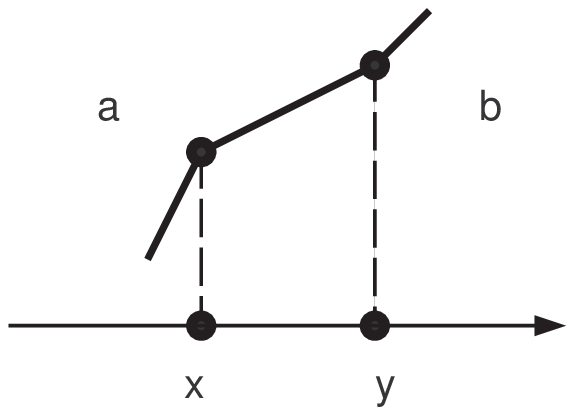}
\captionof{figure}{The configuration on $\G\cong\Zd$ associated with an element of $\wt F$.} \label{fig:config}
\end{center}

\begin{rem}
Obviously, two different transformations $g_1,g_2\in\wt F$ give rise to the same configuration if and only if $g_2=g_1+C$ for $C\in\Z$ (for, if the difference between two functions from $\wt F$ is a constant, then this constant must be an integer, see \secref{sec:slope}). On the other hand, not every configuration $\f\in\fun(\G,\Z)$ corresponds to an element $g\in\wt F$. There are two conditions, whose combination is necessary and sufficient for that. The first condition is
$$
\sum_{\g\in\G} \f(\g) = 0
$$
(for, any $g\in\wt F$ has slope 1 both at $-\infty$ and at $+\infty$, so that the sum of logarithmic increments of the slope must be 0). It guarantees that there exists a transformation $g\in PLF(\R)$ with
$$
\begin{cases}
g(\t)=\t+C_- \qquad \text{at}\; -\infty \;, \\ 
g(\t)=\t+C_+ \qquad \text{at}\; +\infty \;,
\end{cases}
$$
where the difference $C_+-C_-$ is uniquely determined by the configuration $\f$. Now, the second condition is $C_+-C_-\in\Z$ (cf. \secref{sec:slope}). For instance, the configuration $\f=-\d_0+\d_1$ defined as
$$
\f(\g) =
\begin{cases}
-1 \;, &\g=0 \;, \\
+1 \;, &\g=1 \;, \\
\mbox{\;\;\,0} \;, &\text{othwerwise}
\end{cases}
$$
satisfies the first condition, but not the second one; the resulting transformation $g\in PLF(\R)$ is (up to an additive constant)
$$
g(\g) =
\begin{cases}
\g \;, & \g\le 0 \;, \\
\frac{\g}2\;, & 0\le \g \le 1 \;, \\
\g-\frac12 \;, & \g\ge 1 \;.
\end{cases}
$$
\end{rem}

\medskip

\subsection{Composition of configurations}

By the chain rule
$$
(g_1 g_2)'(\g) = g_1'(\g) \cdot g_2'(g_1(\g))
$$
(keep in mind that the group multiplication \eqref{eq:comp} in $\wt F$ is defined by using the ``inverse composition''). Thus,
\begin{equation} \label{eq:cn}
\C_{g_1 g_2}(\g) = \C_{g_1}(\g) + \C_{g_2} (g_1(\g)) \qquad\forall\, g_1,g_2\in\wt F, \g\in\G\;,
\end{equation}
or, in terms of the action \eqref{eq:act},
\begin{equation} \label{eq:cn2}
\C_{g_1 g_2} = \C_{g_1} + \mathbf{S}^{g_1} \C_{g_2} \qquad\forall\, g_1,g_2\in\wt F\;.
\end{equation}

\begin{rem}
Formula \eqref{eq:cn2} looks very similar to the ``configuration component'' of formula \eqref{eq:mult} for the multiplication in the lamplighter groups. Note, however, that the actions $\mathbf{T}$ \eqref{eq:act0} and $\mathbf{S}$ \eqref{eq:act} on the respective configuration spaces, which appear in these formulas, are quite different. The action $\mathbf{T}$ is determined by the action of the group on itself on the \emph{left}, whereas the action $\mathbf{S}$ is determined by the action on $\G\cong\Zd$ on the \emph{right}. In particular, as a result of this
$$
\mathbf{T}^g \d_h = \d_{gh} \qquad\forall\,g,h \in G
$$
in the lamplighter setup of \secref{sec:lamp}, whereas for the action \eqref{eq:act}
$$
\mathbf{S}^g \d_\g = \d_{\g.g^{-1}} \qquad \forall\, g\in\wt F,\; \g\in\G \;.
$$
\end{rem}

Formula \eqref{eq:cn2} allows one to define yet another (``skew'') left action of the group $\wt F$ on the configuration space $\Fun(\G,\Z)$ as
\begin{equation} \label{eq:newact}
(g,\C) \mapsto \C_g + \mathbf{S}^{g}\C \;.
\end{equation}
This action is similar to the natural action of the lamplighter group $\L(G)$ on the configuration space $\Fun(G,\Z/2\Z)$ defined as
$$
(g,\F)\,\C = \F + \mathbf{T}^g\C \;.
$$
It is this action that we mentioned at the end of \secref{sec:stablamp} saying that it is determined by the ``configuration part'' of the multiplication formula in the group $\L(G)$. Note, however, that in the lamplighter case the component $g$ of a group element $(g,\F)\in\L(G)$ is responsible just for ``translating'' a configuration $\C$, whereas the component $\F$ is responsible just for adding an ``increment'' to $\C$. On the contrary, there is no such ``splitting'' in the case of the action \eqref{eq:newact}.

\begin{lem} \label{lem:fix}
The action \eqref{eq:newact} of the group $\wt F$ on $\Fun(\G,\Z)$ has no fixed points.
\end{lem}

\begin{proof}
Let $\C\in\Fun(\G,\Z)$ be a fixed point of the action, i.e.,
$$
\C = \C_g + \mathbf{S}^{g}\C \qquad \forall\,g\in\wt F \;.
$$
Since the configuration $\C_{\wt A}$ associated with the generator $\wt A$ of the group~$\wt F$ (see \secref{sec:change}) is empty, the configuration $\C$ must be then invariant with respect to the transformation $\mathbf{S}^{\wt A}$ determined by the action $\mathbf{S}$ \eqref{eq:act}. In other words, $\C$ must be periodic on $\G\cong\Zd$ with period~1.

In the same way, by taking $g=\wt B$ we arrive at the condition
$$
\C = \C_{\wt B} + \mathbf{S}^{\wt B}\C \;,
$$
where
$$
\C_{\wt B} = -\d_0 + \d_2
$$
by the definition of the generator $\wt B$ (see \secref{sec:change}). Since $\wt B$ acts trivially on $\Zd\cap(-\infty,0]$, it implies that $\C$ can not be 1-periodic, whence a contradiction.
\end{proof}

\subsection{Stabilization of configurations}

\begin{thm} \label{thm:conv}
Let $\mu$ be a finitely supported strictly non-degenerate probability measure on the group $\wt F$. Then for a.e.\ sample path $(g_n)$ of the random walk $(\wt F,\mu)$ the configurations $\C_n=\C_{g_n}$ pointwise converge to a \textsf{limit $\Z$-valued configuration} $\C_\infty\in\Fun(\G,\Z)$, and the space $\Fun(\G,\Z)$ endowed with the arising \textsf{limit distribution} $\la$ is a non-trivial $\mu$\textsf{-boundary}. Therefore, the Poisson boundary $(\p_\mu \wt F,\nu)$ itself is also non-trivial.
\end{thm}

\begin{proof}
By formula \eqref{eq:cn}, for any $\g\in\G$
$$
\C_{n+1} (\g) = \C_n (\g) + \C_{h_{n+1}}(\g.g_n) \;,
$$
where, as usual, $(h_n)$ is the sequence of increments of the random walk~$(g_n)$. Since $\supp\mu$ is finite, the supports of all configurations $\C_h,\,h\in\supp\mu$, are contained in a certain finite set $A\subset\G$. Therefore, by \thmref{thm:transall} the sequence $(\C_n(\g))$ almost surely stabilizes for any $\g\in\G$, i.e., it converges to the limit configuration
$$
\C_\infty = \lim_{n} \C_n = \C_{h_1} + \mathbf{S}^{g_1}\C_{h_2} + \mathbf{S}^{g_2}\C_{h_3} + \dots \;.
$$
Then the resulting measure $\nu$ on the space $\Fun(\G,\Z)$ is $\mu$-stationary with respect to the action \eqref{eq:newact}, cf. \secref{sec:stablamp} for the lamplighter case.

Now it remains to show that the limit configuration $\C_\infty$ can not be the same for a.e.\ sample path $(g_n)$, i.e., that there does not exist a configuration $\C\in\Fun(\G,\Z)$ such that almost surely $\C_\infty=\C$. Indeed, since $\sgr\mu=\wt F$, if this were the case, then the limit configuration $\C$ would have necessarily been a fixed point of the action \eqref{eq:newact} of the group $\wt F$ on $\Fun(\G,\Z)$. However, this is impossible by \lemref{lem:fix}.
\end{proof}

\begin{rem} \label{rem:atom}
The limit distribution $\la$ constructed in \thmref{thm:conv} is concentrated on the subset of $\Fun(\G,\Z)$ which consists only of infinitely supported configurations. The reason is that its complement $\fun(\G,\Z)$ is countable, whereas it is well-known that any non-trivial $\mu$-boundary is purely non-atomic, see Kaimanovich \cite{Kaimanovich95}.
\end{rem}

\section{Geometry of the Schreier graph $\G$} \label{s:geom}

\subsection{The graph $\G$} \label{sec:rays}

Our proof of \thmref{thm:transimple} on the transience of the simple random walk on the Schreier graph $\G\cong\Zd$ of the group~$\wt F$ was based on purely probabilistic considerations, and this is the argument we had referred to in \cite{Kaimanovich04p}.

Shortly thereafter Savchuk independently analyzed the geometry of the Schreier graph of the original Thompson group $F$ (endowed with the standard set of generators) on the dyadic-rational orbit in the interval $[0,1]$, and obtained its complete description \cite{Savchuk10} (also see \cite{Savchuk15} for the Schreier graphs on the other orbits of the canonical action of the group $F$). Since the dyadic-rational orbit of our group $\wt F$ is precisely the image of the dyadic-rational orbit of the original group $F$ under the coordinate change \eqref{eq:t} (see \secref{sec:change}), this Schreier graph is isomorphic to our graph $\G$.

\figref{fig:gr} is a somewhat modified version of the picture of the graph $\G$ which first appeared in \cite{Savchuk10} (I am most grateful to Dmytro Savchuk for sharing his graphics source files with me).

\begin{figure}
\psfrag{o1}[][]{$\frac1{16}$}
\psfrag{o2}[][]{$\frac18$}
\psfrag{o3}[][]{$\frac3{16}$}
\psfrag{o4}[][]{$\frac14$}
\psfrag{o5}[][]{$\frac5{16}$}
\psfrag{o6}[][]{$\frac38$}
\psfrag{o7}[][]{$\frac7{16}$}
\psfrag{o8}[][]{$\frac12$}
\psfrag{o9}[][]{$\frac9{16}$}
\psfrag{o10}[][]{$\frac58$}
\psfrag{o11}[][]{$\frac{11}{16}$}
\psfrag{o12}[][]{$\frac34$}
\psfrag{o13}[][]{$\frac{13}{16}$}
\psfrag{o14}[][]{$\frac78$}
\psfrag{o15}[][]{$\frac{15}{16}$}
\psfrag{o16}[][]{$1$}
\psfrag{e1}[][]{$\frac98$}
\psfrag{e2}[][]{$\frac54$}
\psfrag{e3}[][]{$\frac{11}8$}
\psfrag{e4}[][]{$\frac32$}
\psfrag{e5}[][]{$\frac{13}8$}
\psfrag{e6}[][]{$\frac74$}
\psfrag{e7}[][]{$\frac{15}8$}
\includegraphics{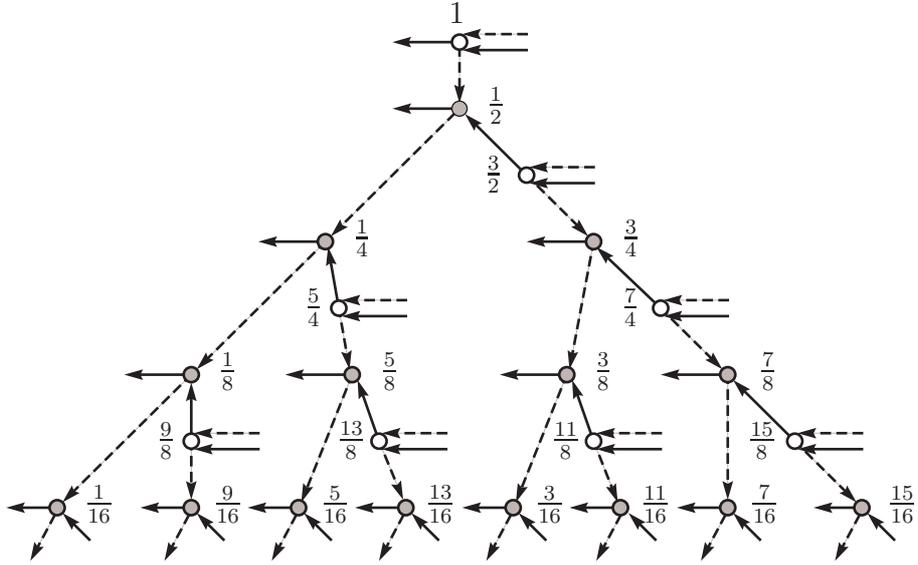}
\caption{The Schreier graph $\G$ of Thompson's group $F\cong\wt F$ on the dyadic-rational orbit in $\R$.} \label{fig:gr}
\end{figure}

\begin{figure}
\psfrag{o}[][]{$\g$}
\psfrag{o1}[][]{$\g-1$}
\psfrag{o2}[][]{$\g-2$}
\psfrag{o3}[][]{$\g-3$}
\includegraphics{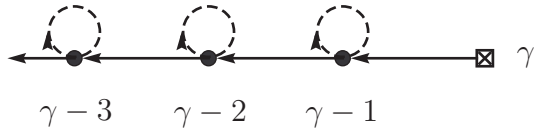}
\caption{Negative type rays in the Schreier graph $\G$.} \label{fig:neg}
\end{figure}

\begin{figure}
\psfrag{o}[][]{$\g$}
\psfrag{o1}[][]{$\g+1$}
\psfrag{o2}[][]{$\g+2$}
\psfrag{o3}[][]{$\g+3$}
\includegraphics{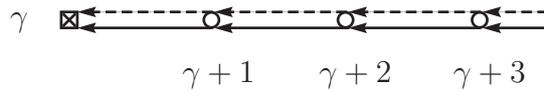}
\caption{Positive type rays in the Schreier graph $\G$.} \label{fig:pos}
\end{figure}

The labelling of the vertices of $\G$ on our \figref{fig:gr} corresponds to the action of $\wt F$ on $\G\cong\Zd$ rather than to the original action of the group $F$ on $\Zd\cap (0,1)$ as in \cite{Savchuk10}. Following \cite{Savchuk10}, we use three different colors for the vertices of the graph $\G$ in order to better exhibit its structure: black for $\g\in\Zd\cap (-\infty,0]$, grey for $\g\in\Zd\cap (0,1)$, and white for $\g\in\Zd\cap [1,\infty)$.

The solid (respectively, dashed) arrows represent the oriented edges in $\G$ corresponding to the generator $\wt A$ (respectively, $\wt B$), see \secref{sec:change}. As is customary for Cayley and Schreier graphs, labels are assigned to \emph{oriented} edges, so that the label of the same edge endowed with the opposite orientation is the letter of the alphabet $\{ \wt A, \wt A^{-1}, \wt B, \wt B^{-1} \}$ inverse to the label of the original edge.

The ``open arrows'' $\leftarrow$ (respectively, $\displaystyle{\substack{\leftarrow\!\!\relbar\\[-.85em] \dashleftarrow }}$) on \figref{fig:gr} represent the infinite \textsf{rays of negative} (respectively, \textsf{positive}) \textsf{type}, see \figref{fig:neg} and \figref{fig:pos}, respectively. These rays are the connected components of the subgraph of~$\G$ with the vertex set $(-\infty,0]\cap\Zd$ (respectively, $[2,\infty)\cap\Zd$) attached to the starting points from the set $(0,1]\cap\Zd$ (respectively, $[1,2)\cap\Zd$). The only vertex of $\G$, to which two rays are attached (both a negative type and a positive type ones), is 1. Note that by the presence of these two families of rays the structure of the graph $\G$ is really begging for applying the coordinate change \eqref{eq:t}.

\subsection{The skeleton} \label{sec:skel}

The \textsf{skeleton} $\ov\G$ of the Schreier graph $\G$ is the tree obtained by removing from $\G$
all the rays of negative and positive types, see \figref{fig:sk}. The tree $\ov\G$ is obviously \emph{roughly isometric} to the usual \emph{rooted binary tree}. Indeed, let $B$ be the 
\textsf{binary tree} which consists just of the grey vertices of $\ov\G$ (i.e., of the points from $(0,1)\cap\Zd$). Then $\ov\G$ and $B$ differ only by the presence of an additional edge (the one between the vertices $1$ and $\frac12$) and of additional white vertices which subdivide some of the edges of $B$ (into two halves.

\bigskip

\begin{center}
\psfrag{o1}[][]{$\frac1{16}$}
\psfrag{o2}[][]{$\frac18$}
\psfrag{o3}[][]{$\frac3{16}$}
\psfrag{o4}[][]{$\frac14$}
\psfrag{o5}[][]{$\frac5{16}$}
\psfrag{o6}[][]{$\frac38$}
\psfrag{o7}[][]{$\frac7{16}$}
\psfrag{o8}[][]{$\frac12$}
\psfrag{o9}[][]{$\frac9{16}$}
\psfrag{o10}[][]{$\frac58$}
\psfrag{o11}[][]{$\frac{11}{16}$}
\psfrag{o12}[][]{$\frac34$}
\psfrag{o13}[][]{$\frac{13}{16}$}
\psfrag{o14}[][]{$\frac78$}
\psfrag{o15}[][]{$\frac{15}{16}$}
\psfrag{o16}[][]{$1$}
\psfrag{e1}[][]{$\frac98$}
\psfrag{e2}[][]{$\frac54$}
\psfrag{e3}[][]{$\frac{11}8$}
\psfrag{e4}[][]{$\frac32$}
\psfrag{e5}[][]{$\frac{13}8$}
\psfrag{e6}[][]{$\frac74$}
\psfrag{e7}[][]{$\frac{15}8$}
\includegraphics{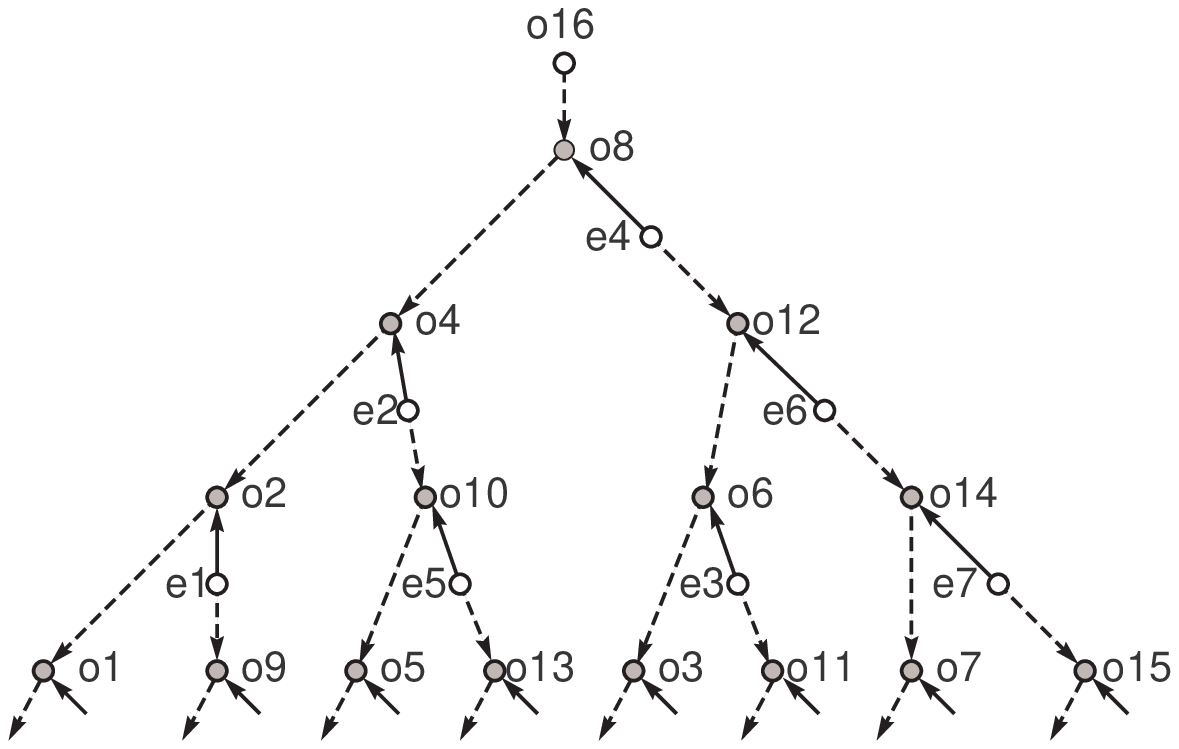}
\captionof{figure}{The skeleton $\ov\G$ of the Schreier graph $\G$.} \label{fig:sk}
\end{center}

\subsection{Transience of $\G$} \label{sec:trans}

Since the rays removed when passing from $\G$ to $\ov\G$ are clearly recurrent for the simple random walk (we have already referred to this property in the proof of \thmref{thm:trans}), the transience of the simple random walk on the Schreier graph is then a direct consequence of the classically known transience of the rooted binary tree (e.g., see Woess \cite{Woess00}). This argument is self-evident, once the geometry of the Schreier graph $\G$ has been exhibited. Although the statement about the transience of $\G$ does not explicitly appear in Savchuk's paper \cite{Savchuk10}, its author was aware of it and mentioned it in some of his talks given at that time.

\subsection{Space of ends of $\G$}

Let $\p\G$ denote the \textsf{space of ends} of the graph $\G$. It obviously splits into a disjoint union
\begin{equation} \label{eq:union}
\p\G = \p\ov\G \cup \p_-\G \cup \p_+\G \;.
\end{equation}
Here $\p_-\G$ and $\p_+\G$ are the infinite sets of ends determined by the rays of the negative and the positive types, respectively, and $\p\ov\G$ is the space of ends of the skeleton $\ov\G$.

The sets $\p_-\G$ and $\p_+\G$ can be identified with the respective subsets of the skeleton $\ov\G$, to which the corresponding rays are attached, i.e., with $\Zd\cap (0,1]$ and with $\Zd\cap [1,2)$, see \secref{sec:rays}. Note that the topology of the union \eqref{eq:union} is the same as that of the \emph{end compactification} of the skeleton $\ov\G$ (with the only difference that the point $1\in\ov\G$ appears with multiplicity 2).

In what concerns $\p\ov\G$, as a topological space it can be easily identified with the \textsf{space $\Z_2^{(1)}=\Z_2\setminus 2\Z_2$ of 2-adic integers of norm 1}. For doing that let us first notice that the downward branchings in the ``grey'' binary tree $B$ (see \secref{sec:skel}) consist in applying the maps
$$
f_0:\g\mapsto\frac{\g}2 \quad\text{and}\quad f_1:\g\mapsto\frac{\g}2+\frac12 \;,
$$
or, equivalently,
$$
f_\e:\g \mapsto \frac{\g}2+\e\frac12 \quad\text{with}\quad \e\in\{0,1\} \;.
$$
Therefore, any geodesic ray $\boldsymbol{\g}=(\g_1,\g_2,\dots)$ in the tree $B$ issued from the point $\g_1=\frac12$ can be encoded by the corresponding sequence $\boldsymbol{\e}=(\e_1,\e_2,\dots)$, where $\e_n$ are uniquely determined by the condition
$$
\begin{aligned}
\g_{n+1} &= f_{\e_n} (\g_n) \\
&= \frac{\e_n}2 + \frac{\e_{n-1}}4 + \frac{\e_{n-2}}8 + \dots + \frac{\e_2}{2^{n-1}} + \frac{\e_1}{2^n} + \frac1{2^{n+1}} \;, \qquad n\ge 1 \;,
\end{aligned}
$$
so that
$$
|\g_n|_2 = 2^n \qquad\forall\,n\ge 1 \;.
$$
Then the sequence
$$
|\g_n|_2 \cdot \g_n = 1 + \e_1 \cdot 2 + \e_2 \cdot 4 + \dots + \e_{n-1} \cdot 2^{n-1}
$$
obviously converges in the 2-adic topology to the 2-adic integer
$$
1 + \sum_{n=1}^\infty \e_n \cdot 2^n
$$
of norm 1. Conversely, any 2-adic integer of norm 1 gives rise to the corresponding geodesic $\boldsymbol{\g}=(\g_1,\g_2,\dots)$ in the tree $B$, i.e., to an end of $\ov\G$.

\subsection{Sections of the end bundle}

By
$$
\Fun(\G,\p\G)\cong(\p\G)^\G
$$
we shall denote the space of \textsf{$\p\G$-valued configurations} on $\G$, or, equivalently, of \textsf{sections of the end bundle}
$$
\G\times\p\G\to\G \;,\qquad (\g,\w)\to\g \;,
$$
over the graph $\G$ (cf. the definitions of the \emph{hyperbolic boundary bundles} and of the \emph{Poisson bundles} in author's papers \cite{Kaimanovich04} and \cite{Kaimanovich05a}, respectively). One can also consider the configurations from $\Fun(\G,\p\G)$ as \textsf{end fields} on $\G$. In the same way as the space $\Fun(\G,\Z)$ of $\Z$-valued configurations on $\G$, the space $\Fun(\G,\p\G)$ is endowed with the left action \eqref{eq:act} of the group $\wt F$.

\subsection{Partial isometries of $\G$} \label{sec:partial}

Generally speaking, groups do not act on their Schreier graphs by graph automorphisms (here and below when talking about graph automorphisms or isomorphisms we mean the maps which preserve the Schreier labelling of edges). In particular, in our case the graph $\G$ and its skeleton $\ov\G$ are both \textsf{rigid}, i.e., their groups of automorphisms are trivial. However, in spite of that the graph $\G$ still has certain symmetry properties sufficient for our purposes. Namely, it has a rather rich set of \textsf{partial isomorphisms}, i.e., of graph isomorphisms between its subgraphs.

First of all, obviously, all the rays of negative (respectively, of positive) type are pairwise isomorphic. Further, for a tree $T$ and any two vertices $o\neq o'\in T$ let $T_{o\to o'}$ denote the \textsf{shadow} of the vertex $o'$ as seen from the vertex $o$, i.e., the subtree of $T$ which consists of all vertices $x\in T$ such that $o'$ lies on the geodesic $[o,x]$. Then the skeleton $\ov\G$ is \textsf{self-similar} in the sense that the shadows $\ov\G_{1\to\frac12}$ and $\ov\G_{1\to\g}$ are isomorphic for any $\g\in\Zd\cap(0,1)$ (i.e., for any ``grey'' vertex $\g\in\ov\G$). If now $\G_{1\to\g}$ denotes the subgraph of $\G$ obtained by reattaching all the negative and positive type rays to the vertices of the shadow $\ov\G_{1\to\g}\subset\ov\G$, then the subgraphs $\G_{1\to\g}\subset\G$ are all pairwise isomorphic for $\g\in\Zd\cap(0,1)$.

\section{Non-trivial behaviour at infinity determined by the boundary of the Schreier graph} \label{s:ends}

\subsection{Mishchenko's work}

In a recent preprint \cite{Mishchenko15} Mishchenko gives yet another proof of the transience of the Schreier graph $\G$ different from the proofs described in \secref{sec:srw} and \secref{sec:trans} (it is based on an explicit estimate of the Dirichlet norm on $\G$). He further notices that due to a specific geometry of $\G$ this transience implies existence of a non-trivial behaviour at infinity ($\equiv$ non-triviality of the Poisson boundary) for the simple random walk on $\G$, and therefore for the simple random walk on Thompson's group $F$ as well (cf. the discussion of the relationship between the simple random walks on $\G$ and on $F\cong\wt F$ in \secref{sec:srw}). Thus, the Poisson boundary of the simple random walk on Thompson's group $F$ is non-trivial \cite[Theorem~2.5]{Mishchenko15}, which provides another proof of our result on the absence of the Liouville property for the group $F$.

The aforementioned geometric argument in \cite{Mishchenko15} is based on the following observation (Theorem~2.3), which we shall quote here in a slightly modified form:

\begin{quote}\small
If a tree $T$ has the property that there exists a vertex $o\in T$ of degree at least 2, and such that for any neighbour $o'$ of $o$ the shadow $T_{o\to o'}$ is transient, then the Poisson boundary of the simple random walk on $T$ is non-trivial.
\end{quote}

In fact, it is classically known (e.g., see Woess \cite{Woess00}) that the transience of the simple random walk on a connected graph implies convergence of its sample paths to an \emph{end} of this graph. Moreover, the same is true for any bounded range random walk, and, as we have already established in \thmref{thm:transall}, the random walk on $\G$ determined by any strictly non-degenerate probability measure $\mu$ on $\wt F$ is transient.

Yet another point is that the boundary behaviour provided by the above argument from \cite{Mishchenko15}
depends on the starting point $\g\in\G$ of the induced random walk, so that the resulting quotient of the Poisson boundary is not a $\mu$-boundary. Actually, the quotient of the Poisson boundary produced by \cite[Theorem 2.3]{Mishchenko15} is finite (it can be identified with the set of neighbours of the vertex $o$), so that it can not be a $\mu$-boundary already for this reason (for, as we have already mentioned in \remref{rem:atom}, any non-trivial $\mu$-boundary is purely non-atomic \cite{Kaimanovich95}).

\subsection{Non-trivial $\mu$-boundary}

\begin{thm} \label{thm:conv2}
Let $\mu$ be a strictly non-degenerate finitely supported probability measure on the group $\wt F$. Then for a.e.\ sample path $\gg=(g_n)$ of the random walk $(\wt F,\mu)$ and any point $\g\in\G\cong\Zd$ the sequence $\g.g_n$ converges to an end $\w_\g=\w_\g(\gg)\in\p\G$, which gives rise to a map
$$
\gg \mapsto (\w_\g)_{\g\in\G}\in\Fun(\G,\p\G)
$$
from the path space $(\wt F^{\Z_+},\P)$ of the random walk $(\wt F,\mu)$ to the configuration space $\Fun(\G,\p\G)$, and the space $\Fun(\G,\p\G)$ endowed with the arising probability measure $\la$ (the image of the measure $\P$ on the path space under the above map) is a non-trivial $\mu$-boundary.
\end{thm}

\begin{proof}
As we have already noticed, \thmref{thm:transall} in combination with the fact that the support of $\mu$ is finite implies the convergence
$$
\g.g_n\to\w_\g\in\p\G \qquad \text{for all}\; \g\in\G \; \text{and a.e.\ sample path}\; (g_n) \;.
$$
Thus, it only remains to show that the distribution of the configurations $(\w_\g)_{\g\in\G}$ is not a point measure.

First of all let us notice that if the configuration $(\w_\g)$ were the same for a.e.\ sample path $(g_n)$, then it would necessarily have been constant (i.e., taking the same value at all points of $\G$), because the group $\wt F$ acts on $\Fun(\G,\p\G)$ by ``changing'' the $\G\cong\Zd$ variable, and $\sgr\mu=\wt F$.

Now let us assume that $\w_\g\equiv\w\in\p\G$. It means that for any starting point $\g\in\G$ almost every sample path of the random walk $(\G,\mu)$ \eqref{eq:rwg} converges to $\w$. Let us denote by $\Q_\#$ the measure on the space of sample paths of this random walk whose initial distribution is the counting measure $\#$ on $\G$.

If $\w\in\p_-\G\cup\p_+\G$ is the end corresponding to a ray of the negative or of the positive type, then there is a $\Q_\#$-non-negligible set of sample paths which are entirely contained in this ray. However, since all the rays of the same type are isomorphic (see \secref{sec:partial}), the same would be true for any other ray of the same type, which means that with positive probability the limit end would be different from $\w$.

If $\w\in\p\ov\G$ is a skeleton end, then, in the same way, the set of the sample paths entirely contained in the subgraph $\G_{1\to\frac12}$ is $\Q_\#$-non-negligible. In view of the self-similarity of the graph $\G$ (see \secref{sec:partial}) it implies that the same would be true for any subgraph $\G_{1\to\g}$. However, one can obviously choose $\g$ in such a way that $\w\notin\p\G_{1\to\g}$.
\end{proof}

The construction of a $\mu$-boundary of the group $F$ on the configuration space $\Fun(\G,\Z)$ from \thmref{thm:conv} heavily used the specifics of this group. On the contrary, the above construction of a $\mu$-boundary on the space of sections $\Fun(\G,\p\G)$ of the end bundle over the Schreier graph~$\G$ is much more general.

\begin{thm} \label{thm:gen}
If a Schreier graph $\G$ of a finitely generated group $G$ is transient (with respect to the simple random walk on it), then for any strictly non-degenerate finitely supported probability measure $\mu$ on $G$ either
\begin{itemize}
\item[(i)]
there exists an end $\w\in\p\G$ such that for any $\g\in\G$ a.e.\ sample path of the induced random walk $(\G,\mu)$ issued from $\g$ converges to $\w$,
\end{itemize}
or
\begin{itemize}
\item[(ii)]
the space of sections $\Fun(\G,\p\G)$ of the end bundle over $\G$ is a non-trivial $\mu$-boundary, so that the Poisson boundary of the random walk $(G,\mu)$ is also non-trivial.
\end{itemize}
\end{thm}

\begin{proof}
In the same way as in the proof of \thmref{thm:transall}, the transience of the simple random walk on $\G$ implies the transience of the induced random walk $(\G,\mu)$ for any strictly non-degenerate probability measure~$\mu$ on $G$. Further, if $\mu$ is finitely supported, then a.e.\ sample path of the induced random walk on $\G$ converges to a random end $\w\in\p\G$. Let us denote by $\{\k_\g\}_{\g\in\G}$ the family of the corresponding \textsf{hitting distributions} on the space of ends $\p\G$. By the strict non-degeneracy of $\mu$, all the measures $\k_\g$ are pairwise equivalent. Now, in case (i) all measures $\k_\g$ coincide with the delta-measure $\d_\g$. Otherwise, in case (ii), all measures $\k_\g$ are not delta-measures, so that $\Fun(\G,\p\G)$ endowed with the arising limit distribution is a non-trivial $\mu$-boundary in the same way as in \thmref{thm:conv2}.
\end{proof}

\subsection{Convergence to components of the end space}

The key ingredient of our proof of \thmref{thm:conv2} was the observation that for a finitely supported measure $\mu$ on $\wt F$ the transience of the random walk $(\G,\mu)$ \eqref{eq:rwg} implies the convergence of a.e.\ sample path to a random end of the graph $\G$ (also see \thmref{thm:gen}). By using the homomorphisms
$$
\chi_a,\chi_b:F\cong\wt F\to\Z
$$
introduced in \secref{sec:f} one can easily describe the components of the decomposition \eqref{eq:union} of the space of ends $\p\G$, on which the hitting measures $\k_\g$ of the random walk $(\G,\mu)$ are actually concentrated.

Since the measure $\mu$ is finitely supported, the restriction of the random walk $(\G,\mu)$ to any of the negative (respectively, positive) type rays (see \secref{sec:rays}) coincides, outside of a finite neighbourhood of ray's origin, just with the usual translation invariant random walk on $\Z$. Here and below we identify the negative (respectively, positive) type rays in $\G$ with the \textsf{negative} (respectively, \textsf{positive}) \textsf{integer ray} $\Z_-=\{\dots,-2,-1,0\}$ (respectively, $\Z_+=\{0,1,2,\dots\}$) of $\Z$. As it follows from the description of the graph $\G$ (see \figref{fig:gr}), the step distribution of the induced random walk on $\Z_-$ (respectively, on $\Z_+$) is $(-\chi_a)(\mu)$ (respectively, $(-\chi_a-\chi_b)(\mu)$), where by $\chi(\mu)$ we denote the image of the measure $\mu$ under a group homomorphism $\chi:\wt F\to\Z$. Therefore, the drift of the induced random walk on $\Z_-$ ($\equiv$ on the negative type rays) is the barycentre
$$
\ov{(-\chi_a)(\mu)} = -\a \;,
$$
and the drift of the induced random walk on $\Z_+$ ($\equiv$ on the positive type rays) is the barycentre

$$
\ov{(-\chi_a-\chi_b)(\mu)} = -\a -\b \;,
$$
where
$$
\a = \ov{\chi_a(\mu)} \quad\text{and}\quad \b = \ov{\chi_b(\mu)}
$$
denote the barycentres (the expectations) of the measures $\chi_a(\mu)$ and $\chi_b(\mu)$ on $\Z$, respectively.

It is classically known that the recurrence properties of the random walk on $\Z$ governed by a finitely supported step distribution $\s$ are completely determined by the barycentre $\ov\s$. In particular, the ray $\Z_-$ (respectively, $\Z_+$) is \textsf{transient} (i.e., the random walk escapes to infinity along this ray) if and only if the drift $\ov\s$ is negative (respectively, positive). Therefore, we arrive at the following description of the components of the decomposition \eqref{eq:union}, on which the hitting measures $\k_\g$ are concentrated, in terms of the parameters $\a,\b\in\R$ (also see \figref{fig:ab}):

\begin{prp} \label{prp:cases}
Under the conditions of \thmref{thm:conv2}
\begin{itemize}
\item[(i)] If $\a>0$ and $\a+\b<0$, then both the negative type and the positive type rays are transient, and the hitting measures $\k_\g$ are concentrated on the union $\p_-\G \cup \p_+\G$;
\item[(ii)] If $\a>0$ and $\a+\b\ge 0$, then only the negative type rays are transient, and the hitting measures $\k_\g$ are concentrated on $\p_-\G$;
\item[(iii)] If $a\le 0$ and $\a+\b<0$, then only the positive type rays are transient, and the hitting measures $\k_\g$ are concentrated on $\p_+\G$;
\item[(iv)] If $a\le 0$ and $\a+\b\ge 0$, then neither the negative nor the positive type rays are transient, and the hitting measures $\k_\g$ are concentrated on $\p\ov\G$.
\end{itemize}
\end{prp}

\begin{center}
\psfrag{a}[][]{$\alpha$}
\psfrag{b}[][]{$\beta$}
\psfrag{A}[][]{$\p_-\G$}
\psfrag{B}[][]{$\p_-\G \cup \p_+\G$}
\psfrag{C}[][]{$\p_+\G$}
\psfrag{D}[][]{$\p\ov\G$}
\includegraphics{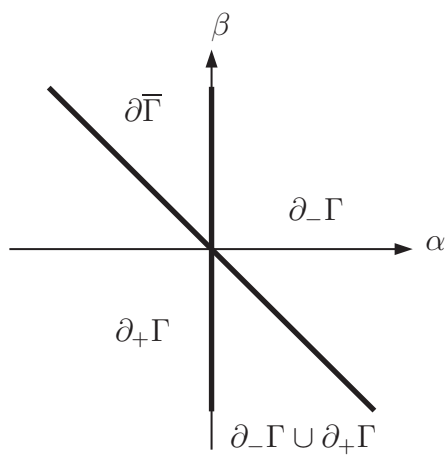}
\captionof{figure}{The domains in the parameter plane $(\a,\b)$ corresponding to different supports of the hitting measures $\k_\g$.} \label{fig:ab}
\end{center}

\bigskip

\prpref{prp:cases} allows us to make more precise the description of the $\mu$-boundary of the random walk $(\wt F,\mu)$ obtained in \thmref{thm:conv2}.

\begin{thm} \label{thm:conv2a}
Under conditions of \thmref{thm:conv2} the arising probability measure $\la$ is concentrated on
$$
\begin{cases}
\Fun(\G,\p_-\G \cup \p_+\G) \;,  & \text{if}\;\a>0, \;\text{and}\; \a+\b<0\;,\\
\Fun(\G,\p_-\G)\;,  & \text{if}\; \a>0, \;\text{and}\; \a+\b\ge 0\;\\
\Fun(\G,\p_+\G)\;,  & \text{if}\; a\le 0, \;\text{and}\; \a+\b<0\;\\
\Fun(\G,\p\ov\G)\;,  & \text{if}\; a\le 0, \;\text{and}\; \a+\b\ge 0\;.
\end{cases}
$$
\end{thm}

\section{Concluding remarks} \label{s:concl}

\subsection{Relaxing conditions on the step distribution}

Throughout the paper we have used the assumption that the measure $\mu$ on the group $\wt F\cong F$ is finitely supported and strictly non-degenerate. It would be interesting to see, to what extent our results can be carried over to more general measures. [Let us remind once again that amenability of the group $F$ is equivalent to existence of a Liouville symmetric probability measure $\mu$ with $\supp\mu=F$.]

Here one can try to relax both the conditions on \emph{how much the measure $\mu$ is allowed to spread out}:
\begin{gather*}
    \text{finite support} \\
    \downarrow \\
    \text{finite first moment} \\
    \downarrow \\
    \text{finite entropy} \\
    \downarrow \\
    \text{general measures}
\end{gather*}
and the conditions on \emph{how non-degenerate the measure $\mu$ is allowed to be}:
\begin{gather*}
    \sgr\mu=\wt F \\
    \downarrow \\
    \gr\mu=\wt F \\
    \downarrow \\
    \text{$\gr\mu$ is a \emph{non-elementary} subgroup of $\wt F$}\;.
\end{gather*}
The meaning of ``non-elementary'' of course needs to be specified in our context (for instance, cf. Kaimanovich -- Masur \cite{Kaimanovich-Masur96} for the subgroups of the mapping class group, or Kaimanovich \cite{Kaimanovich00a} for the subgroups of groups with hyperbolic properties).

As for the conditions on how much the measure $\mu$ is allowed to spread out, one can expect that an extension to the class of measures with a \emph{finite first moment} should be quite feasible (cf. the case of the lamplighter groups considered by the author \cite{Kaimanovich91} and the case of hyperbolic graphs considered by Kaimanovich -- Woess \cite{Kaimanovich-Woess92}). Notice, however, that, once the finite first moment condition is dropped, the ``hyperbolic'' and the ``lamplighter'' models drastically diverge. Boundary convergence in hyperbolic spaces does not require any moment conditions, see Kaimanovich \cite{Kaimanovich00a}, Maher -- Tiozzo \cite{Maher-Tiozzo14p}, whereas
for the lamplighter groups, due to their amenability, the situation is completely different.

For instance, there are examples of measures with infinite first moment on the lamplighter groups, for which the configurations do not stabilize, but the Poisson boundary is still non-trivial, see Kaimanovich \cite{Kaimanovich83a,Kaimanovich91}. In fact, Erschler proved in \cite{Erschler04} that the Poisson boundary on the lamplighter groups is non-trivial for all step distributions with finite entropy provided the quotient random walk is transient. She also described another family of groups (which she calls \emph{groups of Baumslag type}) with this property. These groups are quite similar to the lamplighter groups, but the difference is that for the Baumslag type groups certain infinitely supported configurations are also allowed.

Outside of the class of measures with finite entropy the situation becomes even more complex. There are examples of measures~$\mu$ on the lamplighter groups, for which the Poisson boundary is trivial in spite of non-triviality of the Poisson boundary for the \textsf{reflected measure} $\check\mu(g)=\mu(g^{-1})$, see Kaimanovich \cite{Kaimanovich83a}. This is impossible for the measures with finite entropy because of the entropy criterion of triviality of the Poisson boundary, as in the finite entropy case the asymptotic entropies of the original and of the reflected random walks are obviously the same.

\subsection{Other subgroups of $PLF(\R)$}

It would be interesting to understand, to what extent our results can be carried over to other subgroups of $PLF(\R)$, for instance, to the group generated by the translation $\g\mapsto\g+1$ and the transformation defined as

$$
\g \mapsto
\begin{cases}
\g \;, & \g\le 0 \;, \\
2\g \;, & \g\ge 0 \;.
\end{cases}
$$
Note that by a result of Brin -- Squier \cite{Brin-Squier85} $PLF(\R)$ (and therefore all its subgroups as well) does not contain non-abelian free subgroups.

As it has been pointed out by the referee, our ``drifting'' technique used in the proof of \thmref{thm:trans} should be applicable to the ``$F$-series'' of the so-called \emph{Thompson--Stein groups} (they are a generalization of Thompson's groups introduced by Stein \cite{Stein92}). In particular, it almost \emph{verbatim} carries over to the groups $F(p)$ introduced by Burillo -- Cleary -- Stein \cite{Burillo-Cleary-Stein01} (for these groups the break points are allowed to be in $\Z[\frac1p]$ for an integer $p>1$). An interesting feature of the family $\{F(p)\}$ is that all these groups embed one into the other, so that they are all either amenable or non-amenable simultaneously.

We are not aware of any explicit description of the Schreier graphs of the canonical action of any of the Thompson--Stein groups other than Savchuk's description of the Schreier graphs for the original Thompson's group $F$ \cite{Savchuk10,Savchuk15}.

\subsection{The lattice of $\mu$-boundaries and the problem of identification of the Poisson boundary}

Although the $\mu$-boundaries constructed in \thmref{thm:conv} and \thmref{thm:conv2} seem to be quite different, currently we do not have any information about their mutual position in the lattice of all $\mu$-boundaries. In particular, \emph{a priori} it is not excluded that both these $\mu$-boundaries might in fact coincide with the full Poisson boundary (which is the maximal element in the lattice of $\mu$-boundaries). If not, what are, for instance, their infimum and supremum? Can one be just a quotient of the other one? And, of course, how are they related to the full Poisson boundary? Will their supremum, one of them, or maybe even both be the Poisson boundary? See author's papers \cite{Kaimanovich96,Kaimanovich00a} for a discussion of the general identification problem for the Poisson boundary.

\subsection{Configuration spaces as $\mu$-boundaries, and generating partitions of the Poisson boundary.}

The $\mu$-boundary from \thmref{thm:conv2} seems to be the first example when a boundary behavior is realized on the space of configurations on a single group orbit in such a way that
the group action consists of just permuting the configurations without changing their point values.

There is a well-known construction in ergodic theory which allows one to realize general group actions as actions on configuration spaces by translations (e.g., see Cornfeld -- Fomin -- Sinai  \cite{Cornfeld-Fomin-Sinai82}). Namely, let $(X,m)$ be a Lebesgue measure space endowed with a measure class preserving action of a countable group $G$. Given an at most countable measurable partition $\xi$, let $X_\xi$ be the corresponding countable quotient space (i.e., the space of the elements of the partition $\xi$), and let $x\mapsto \xi(x)$ be the corresponding quotient map which assigns to any point $x\in X$ the element of the partition $\xi$ which contains $x$. The function
$$
\Phi_x: G\to X_\xi \;,\quad g \mapsto \xi(g^{-1} x) \;,
$$
can be then considered as an $X_\xi$-valued configuration on $G$, and the map $x\mapsto \Phi_x$ is obviously $G$-equivariant if one endows the space of configurations on $G$ with the action \eqref{eq:act0}. If the map $x\mapsto\Phi_x$ separates (mod 0) the points of $X$, i.e., if this map is an isomorphism between the space $(X,m)$ and the space of $X_\xi$-valued configurations on $G$, then the partition $\xi$ is called \textsf{generating}.

Generating partitions for $\Z$-actions are a classical tool of ergodic theory. Moving to more general actions, it is, for example, easy to see that if one takes the partition of the boundary of a free group into the cylinder sets determined by the first letter of the infinite words representing boundary points, then this partition is generating even in the topological category. Therefore, this example provides a symbolic realization of the action of the free group on its Poisson boundary whenever the Poisson boundary can be identified with the space of infinite words (for instance, for the step distributions with a finite first moment, see Kaimanovich \cite{Kaimanovich00a}). However, we are not aware of any work on generating partitions for Poisson boundaries or their quotients in general.

\subsection{Poisson boundary and sections of boundary bundles over Schreier graphs}

The statement of \thmref{thm:gen} on the construction of a non-trivial $\mu$-boundary for a group $G$ from the space of ends of its transient Schreier graph $\G$ is, of course, very well known when one takes for $\G$ just the group $G$ (in which case the action of the group $G$ on itself extends to the space $\p G$, and the limit configurations are equivariant as functions from $G$ to $\p G$). However, it would be interesting to have other examples in the situation when the stabilizers $G_\g,\;\g\in\G$, are sufficiently far from being normal. It is natural to impose the condition that the Schreier graph is determined by a \textsf{core-free subgroup} $H\subset G$ (i.e., that the action of $G$ on the homogeneous space $\G\cong H\backslash G$ is faithful), as otherwise one can always pass to the quotient of $G$ by the normal core of $H$. Note that for Thompson's group $F\cong\wt F$ the latter condition is obviously satisfied as the canonical action of Thompson's group is completely determined by its restriction to the dyadic-rational orbit (cf. \secref{sec:change}).

Instead of the end bundle one can use the \textsf{Poisson bundle} over the graph $\G$, i.e., the map
$\G \times \p_\mu\G \to \G$, where $\p_\mu\G$ denotes the Poisson boundary of the induced random walk $(\G,\mu)$ on the graph $\G$. The space of its sections is the space of configurations $\Fun(\G,\p_\mu\G)$. Let $\bnd=\bnd_{(\G,\mu)}$ denote the \textsf{boundary map} from the path space of the random walk $(\G,\mu)$ to the Poisson boundary $\p_\mu\G$. If $\p_\mu\G$ is non-trivial, then the projection
$$
\gg=(g_n) \mapsto \{\bnd(\g.\gg)\}_{\g\in\G} \;,
$$
where $\gg\mapsto\g.\gg$ is the map \eqref{eq:gg} from the path space of the random walk $(G,\mu)$ to the path space of the induced random walk $(\G,\mu)$, makes $\Fun(\G,\p_\mu\G)$ a non-trivial $\mu$-boundary. It would be interesting to find conditions under which this $\mu$-boundary coincides with the full Poisson boundary of the random walk $(G,\mu)$.

\bibliographystyle{amsalpha}
\bibliography{C:/S/MyTEX/mine}

\end{document}